\documentclass[a4paper]{amsart}
\usepackage{amssymb,amscd}
\usepackage{graphicx}
\usepackage[all]{xy}
\usepackage{amsmath}
\usepackage{cancel}
\usepackage{enumitem}
\usepackage{hyperref}

\input xypic

\newtheorem{theorem}{Theorem}[section]
\newtheorem{proposition}[theorem]{Proposition}
\newtheorem{lemma}[theorem]{Lemma}
\newtheorem{corollary}[theorem]{Corollary}

\theoremstyle{definition}
\newtheorem{definition}[theorem]{Definition}

\theoremstyle{remark}
\newtheorem*{remark}{Remark}

\numberwithin{equation}{section}

\renewcommand{\phi}{\varphi}
\newcommand{\bl}{\mbox{$\lambda\kern-0.53em\lambda$}}
\newcommand{\bmu}{\mbox{$\mu\kern-0.55em\mu$}}
\newcommand{\bnu}{\mbox{$\nu\kern-0.51em\nu$}}
\def\bphi{\mbox{$\varphi\kern-0.59em\varphi$}}

\def\R{\mathbb R}

\def\Z{\mathbb Z}

\def\sK{\mathcal K}

\newcommand{\mb}[1]{{\textbf {\textit#1}}}

\renewcommand{\binom}[2]{{C_{#1}^{#2}}}

\newcommand{\rank}{\mathop{\mathrm{rank}}}

\def\raag{\mbox{\it RA\/}}
\def\racg{\mbox{\it RC\/}}

\newcommand{\rk}{\mathcal R_{\mathcal K}}

\def\pt{\mathit{pt}}

\def \RCK {{\racg_\sK}}

\def \K {\mathcal{K}}

\keywords{right-angled Coxeter group, associated graded Lie algebra, graph, lower central series}

\begin{document}

\title[On consecutive factors of the lower central series of $\racg_\sK$]{On consecutive factors of the lower central series of right-angled Coxeter groups}
\author{Yakov Veryovkin}

\address{\parbox{\linewidth}{International Laboratory of Algebraic Topology and Its Applications,\\National Research University Higher School of Economics, Moscow, Russia;\\Lomonosov Moscow State University, Moscow, Russia}}
\email{verevkin\_j.a@mail.ru}

\author{Temur Rahmatullaev}
\address{\parbox{\linewidth}{International Laboratory of Algebraic Topology and Its Applications,\\National Research University Higher School of Economics, Moscow, Russia;\\Lomonosov Moscow State University, Moscow, Russia}}
\email{temurbek.rakhmatullaev@math.msu.ru}

\thanks{The first author's research was supported by the Russian Science Foundation, project 23-11-00143, \href{https://rscf.ru/en/project/23-11-00143/}{https://rscf.ru/en/project/23-11-00143/}, at the Steklov Mathematical Institute of the Russian
Academy of Sciences (Secs. 1, 2, 4, and 5). The second author’s research was carried out in the
framework of the ``Mirror Laboratories'' project at National Research University Higher School of
Economics (Sec. 3).}

\begin{abstract}
We study the lower central series of the right-angled Coxeter group $\racg_\sK$ and the
corresponding associated graded Lie algebra $L(\racg_\sK)$ and describe the basis of the fourth graded component of $L(\racg_\sK)$ for any $\sK$.
\end{abstract}

\maketitle
\section{Introduction}
A right-angled Coxeter group $\racg_\sK$ is a group with $m$
generators $g_1,\ldots,g_m$ satisfying the relations $g_i^2=1$ for all $i \in \{1,\ldots,m\}$ and the commutation relations
$g_ig_j=g_jg_i$ for some pairs $\{i,j\}$. Each such group can be defined by a graph $\sK^1$ with $m$ vertices, where two vertices are connected by an edge whenever the corresponding generators commute. Right-angled Coxeter groups are classical objects in geometric group theory. In the present paper, we study the lower central series of the right-angled Coxeter groups $\racg_\sK$ and the corresponding associated graded Lie algebra $L(\racg_\sK)$.

For the right-angled Artin groups $\raag_\sK$ (which differ from the right-angled Coxeter groups $\racg_\sK$ in that the relations $g_i^2 = 1$ are absent) the associated Lie algebars $L(\raag_\sK)$ were completely calculated in~\cite{Duch-Krob}, see also~\cite{WaDe},~\cite{Papa-Suci}. Specifically, it was proved that the Lie algebra $L(\raag_\sK)$ is isomorphic to the graph Lie algebra (over $\mathbb Z$) corresponding to the graph $\sK^1$.

For the right-angled Coxeter groups, the quotient groups $\gamma_1(\racg_\sK) / \gamma_n(\racg_\sK)$ were studied in some
special cases for some $n$ (\cite{Struik1},~\cite{Struik2}). For $n \geqslant 4$ these studies encountered some difficulties similar to those we experienced when calculating the successive quotients $\gamma_{n-1}(\racg_\sK) / \gamma_n(\racg_\sK)$. The problem of describing the associated Lie algebra $L(\racg_\sK)$ for the right-angled Coxeter groups is much more
complicated than for the right-angled Artin groups, because the Lie algebra $L(\racg_\sK)$ and the graph Lie
algebra $L_\sK$ over $\mathbb Z_2$ are not necessarily isomorphic (see~\cite[Example 4.3]{ve3}). In~\cite{ve3}, a Lie algebra epimorphism $L_\sK \rightarrow L(\racg_\sK)$ was constructed, its kernel was described in severl cases, and a combinatorial description of the bases of the first three graded components of the Lie algebra was given for an arbitrary group $\racg_\sK$. In~\cite{ve5}, a minimal set of generators in $L^4(\racg_\sK)$ was constructed for the case in whick $\sK$ is the discrete simplicial complex on four points or an arbitrary simplicial complex on three points.

In~\cite{WaldingerLCS}, the dimensions of successive lower central series quotients were calculated and bases in these
quotients were constructed for free products of direct sums of cyclic groups of order 2, which form
a subset of the set of right-angled Coxeter groups. In the present paper, we construct a basis of the fourth
graded component of the associated Lie algebra for Coxeter groups on four points (for the case of two
points, the associated Lie algebra has been completely described; see~\cite[Proposition 4.4]{ve3}) and present
an algorithm for constructing a basis of the fourth graded component of the associated Lie algebra for
Coxeter groups in the general case. In contrast to the bases constructed in~\cite{WaldingerLCS}, our bases completely consist of nested commutators.

The first author is a winner of the 2019 ``Young Mathematics of Russia'' competition. The author
expresses deep gratitude to his scientific supervisor Taras Evgenievich Panov for setting the problem, constant attention, and assistance in the work.

\section{Preliminaries}
Let $\sK$ be an (abstract) simplicial complex on the set $ [m] = \{1,2, \dots, m \}$. A subset $$I=\{i_1,\ldots,i_k\}\in\mathcal K$$ is called a
\emph{simplex} (or a \emph{face}) of~$\sK$. Throughout the paper, we assume that $\sK$ contains $\varnothing$ and all singletons $\{i\}$, $i = 1, \ldots, m$.

By $F_m$ or $F(g_1, \ldots, g_m)$ we denote the free group of rank $m$ with generators $g_1, \ldots, g_m$.

The \emph{right-angled Coxeter (Artin) group} corresponding to $\sK$ is the group $\racg_\sK$ ($\raag_\sK$) defined by generators and relations as follows:
\begin{align*}
  \racg_\sK&= F(g_1,\ldots,g_m)\big/ (g_i^2 = 1 \text{ for } i \in \{1, \ldots, m\}, \; \; g_ig_j=g_jg_i\text{ for
  }\{i,j\}\in\sK),\\
  \raag_\sK&= F(g_1,\ldots,g_m)\big/ (g_ig_j=g_jg_i\text{ for
  }\{i,j\}\in\sK).
\end{align*}

It follows from the definition that the group $\racg_\sK$ ($\raag_\sK$) depends only on the graph~$\sK^1$, i.e., on the
1-skeleton of~$\sK$.

Consider the construction of the polyhedral product. Let $\sK$ be a simplicial complex on the set~$[m]$, and let
\[
  (\mb X,\mb A)=\{(X_1,A_1),\ldots,(X_m,A_m)\}
\]
be a sequence of $m$ pairs of pointed topological spaces, where $\pt\in A_i\subset X_i$. For each subset
$I\subset[m]$, we set
\begin{equation}\label{XAI}
  (\mb X,\mb A)^I=\bigl\{(x_1,\ldots,x_m)\in
  \prod_{k=1}^m X_k\colon\; x_k\in A_k\quad\text{for }k\notin I\bigl\}
\end{equation}
and define the \emph{polyhedral product}
\[
  (\mb X,\mb A)^{\sK}=\bigcup_{I\in\mathcal K}(\mb X,\mb A)^I=
  \bigcup_{I\in\mathcal K}
  \Bigl(\prod_{i\in I}X_i\times\prod_{i\notin I}A_i\Bigl),
\]
where the union is taken inside the product $\prod_{k=1}^m X_k$.

If all pairs $(X_i,A_i)$ are the same, i.e., $X_i=X$ and
$A_i=A$ for all $i=1,\ldots,m$, then we use the notation $(X,A)^\sK$
for $(\mb X,\mb A)^\sK$. Further, if all $A_i$ are $\pt$, then we use the shorthand notation $\mb X^\sK$ for $(\mb X,\pt)^\sK$ and accordingly
$X^\sK$ for $(X,\pt)^\sK$.
For a detailed description of this construction and examples, see~\cite[\S3.5]{bu-pa00},~\cite{b-b-c-g10},~\cite[\S4.3]{bu-pa15}.

Let $(X_i,A_i)=(D^1,S^0)$ for all $i$, where $D^1$ is the closed interval and $S^0$ is the boundary of $D^1$ (a set consisting of two points). The corresponding polyhedral product is known as a \emph{real moment-angle complex}~\cite[\S3.5]{bu-pa00},~\cite{bu-pa15} and is denoted by~$\rk$:
\begin{equation}\label{rk}
  \rk=(D^1,S^0)^\sK=\bigcup_{I\in\sK}(D^1,S^0)^I.
\end{equation}

We also need th polyhedral product $(\R P^\infty)^\sK$, where $\R P^\infty$ is the infinite-dimensional projective space.

A simplicial complex $\sK$ is called a \emph{flag} complex if any set of its vertices pairwise connected by edges is the set of vertices of some simplex. Any flag complex $\sK$ is determined by its $1$-skeleton $\sK^1$.

The relationship between polyhedral products and right-angled Coxeter groups is as follows.

\begin{theorem}[{ \cite[Corollary~3.4]{pa-ve}}]\label{coxfund}
Let $\sK$ be a simplicial complex on $m$ vertices.
\begin{itemize}
\item[(a)] $\pi_1((\R P^\infty)^\sK)\cong\racg_\sK$.
\item[(b)] Each of the spaces $(\R P^\infty)^\sK$ and $\rk$ is aspherical if and only if $\sK$ is a flag complex.
\item[(c)] $\pi_i((\R P^\infty)^\sK)\cong\pi_i(\rk)$ for $i\ge2$.
\item[(d)] The group $\pi_1(\rk)$ is isomorphic to the commutant~$\racg'_\sK$.
\end{itemize}
\end{theorem}

For each subset~$J\subset[m]$, consider the restriction
\[
  \sK_J=\{I\in\sK\colon I\subset J\},
\]
of $\sK$ to~$J$, which is also called a \emph{full subcomplex} of~$\sK$.

The following theorem gives a combinatorial description of the homology of the real moment-angle
complex $\mathcal R_\sK$.

\begin{theorem}[{\cite{bu-pa00}, \cite[\S4.5]{bu-pa15}}]\label{homrk} For any $k\geq 0$, one has an isomorphism
\[
  H_k(\rk;\Z)\cong\bigoplus_{J\subset[m]}\widetilde
  H_{k-1}(\sK_J),
\]
where $\widetilde H_{k-1}(\sK_J)$ is the reduced simplicial homology group of the simplicial complex~$\sK_J$.
\end{theorem}

If $\sK$ is a flag complex, then Theorem~\ref{homrk} also provides a description of the integer homology of the commutant $\racg_\sK'$.

\medskip
Let $G$ be a group. The \emph{commutator} of two elements $a, b \in G$ is defined by the formula $(a,b) = a^{-1}b^{-1}ab$.

The \emph{simple nested commutator} of length $k$ of elements $q_i$ is defined as the nested commutator
$$
(q_{i_1}, q_{i_2}, \ldots, q_{i_k}) := (\ldots((q_{i_1}, q_{i_2}), q_{i_3}), \ldots, q_{i_k}).
$$
Accordingly, in Lie algebras we consider the simple nested commutators
$$
[\mu_{i_1}, \mu_{i_2}, \ldots, \mu_{i_k}] := [\ldots[[\mu_{i_1}, \mu_{i_2}], \mu_{i_3}], \ldots, \mu_{i_k}].
$$

For any group $G$ and any three elements $a, b, c \in G$, one has the \emph{Witt--Hall identities}:
\begin{equation}\label{WH}
\begin{aligned}
(a, bc) &= (a, c) (a, b) (a, b, c),\\
(ab, c) &= (a, c) (a, c, b) (b, c),\\
(a,b,c)(b,c,a)(c,a,b)&=(b,a)(c,a)(c,b)^a(a,b)(a,c)^b(b,c)^a(a,c)(c,a)^b,
\end{aligned}
\end{equation}
where $a^b = b^{-1}ab$.

Let $H, W \subset G$ be subgroups. Then we define a subgroup $(H, W) \subset G$ as the subgroup generated by the commutators $(h, w)$, where $h \in H, w \in W$. In particular, the  \emph{commutant} $G'$ of the group $G$ is $(G, G)$.

For any group $G$, we set $\gamma_1(G) = G$, and further, inductively, $$\gamma_{k+1}(G) = (\gamma_{k}(G), G).$$ The resulting sequence of groups $\gamma_1(G), \gamma_2(G), \ldots, \gamma_k(G), \ldots$ is called the \emph{lower central series} of $G$.

If $H \subset G$ is a normal subgroup, i.e., $H = g^{-1}Hg$ for any $g \in G$, then we use the notation $H \lhd G$.

In particular, $\gamma_{k+1}(G) \lhd \gamma_k(G)$, and the quotient groups $\gamma_{k}(G) / \gamma_{k+1}(G)$ are abelian. We introduce the notation $L^k (G) := \gamma_{k}(G) / \gamma_{k+1}(G)$ and consider the direct sum
$$
L(G) := \bigoplus_{k=1}^{+\infty} L^k (G).
$$
By $\overline{a}_k$ we denote the class of an element $a_k \in \gamma_k(G) \subset G$ in the quotient group~$L^k (G)$. If $a_k \in \gamma_k(G), \; a_l \in \gamma_l(G)$, then $(a_k, a_l) \in \gamma_{k+l}(G)$. Then the Witt--Hall identities imply that $L(G)$ is a graded Lie algebra over $\mathbb{Z}$ (a Lie ring) with Lie brackets defined by the formula $[\overline{a}_k, \overline{a}_l] := \overline{(a_k, a_l)}$. The Lie algebra $L(G)$ is called the \emph{associated Lie algebra} of the group $G$.

\begin{theorem}[{\cite[Theorem 4.5]{pa-ve}}]\label{gscox}
Let $\racg_\sK$ be the right-angled Coxeter group corresponding to a simplicial complex~$\sK$ on $m$ vertices. Then the commutant $\racg'_\sK$ has a finite minimal set of generators consisting of $\sum_{J\subset[m]}\rank\widetilde
H_0(\sK_J)$ nested commutators of the form
\begin{equation}\label{commuset}
  (g_i,g_j),\quad (g_{i},g_j,g_{k_1}),\quad\ldots,\quad
  (g_{i},g_{j},g_{k_1},g_{k_2},\ldots,g_{k_{m-2}}),
\end{equation}
where $i < j > k_1 > k_2 > \ldots > k_{\ell-2}$, $k_s\ne i$ for all~$s$, and $i$ is the least vertex in some connected component not containing $j$ of the full subcomplex $\sK_{\{i, j, k_1,\ldots,k_{\ell-2}\}}$.\end{theorem}

\begin{remark}
In~\cite{pa-ve}, the commutators are nested to the right. In the present paper, by convention, they are
nested to the left.
\end{remark}

From Theorems \ref{homrk} and \ref{gscox}, we obtain the following assertion.
\begin{corollary}\label{h1rk}
The group $H_1(\rk) = \racg_\sK' / \racg_\sK''$ is a free abelian group of rank $$\sum_{J\subset[m]}\rank\widetilde H_0(\sK_J)$$ with a basis consisting of images of iterated commutators described in Theorem~\ref{gscox}.
\end{corollary}

The following standard result holds (see \cite[\S5.3]{Ma-Car-Sol}):

\begin{proposition}\label{comb}
Let $G$ be a group with generators $g_i, i \in I$. The terms $\gamma_k(G)$ of the lower central
series are generated by simple nested commutators of length $\geq k$ of the generators and their
inverses.
\end{proposition}

\begin{proposition}[{\cite[Proposition 3.3]{ve3}}]\label{kv}
The square of any element in $\gamma_k(\racg_\sK)$ lies in $\gamma_{k+1}(\racg_\sK)$.
\end{proposition}

The following results hold as well.

\begin{proposition}[{\cite[the statement after (4.22)]{WaldingerLCS}}]\label{numbergens3}
Let $\sK$ be the discrete simplicial complex on three
points; i.e., $\racg_\sK = \mathbb Z_2\langle g_1 \rangle \ast \mathbb Z_2\langle g_2 \rangle \ast \mathbb Z_2\langle g_3 \rangle$. Then $L^4(\racg_\sK)$ has a minimal set of eight generators.
\end{proposition}

\begin{proposition}[{\cite[the statement after (4.22)]{WaldingerLCS}}]\label{numbergens3oneedge}
Let $\sK$be a simplicial complex on three points with
only one edge; i.e., $\racg_\sK = (\mathbb Z_2\langle g_1 \rangle \oplus \mathbb Z_2\langle g_2 \rangle) \ast \mathbb Z_2\langle g_3 \rangle$. Then $L^4(\racg_\sK)$ has a minimal set of four generators.
\end{proposition}

\begin{proposition}[{\cite[the statement after (4.22)]{WaldingerLCS}}]\label{numbergens4onetriangle}
Let $\sK$ be the simplicial complex on four points
deﬁned as the disjoint union of a point and the boundary of a triangle; i.e., $$\racg_\sK = (\mathbb Z_2\langle g_1 \rangle \oplus \mathbb Z_2\langle g_2 \rangle \oplus \mathbb Z_2 \langle g_3 \rangle) \ast \mathbb Z_2\langle g_4 \rangle.$$ Then $L^4(\racg_\sK)$ has a minimal set of ten generators.
\end{proposition}

\begin{proposition}[{\cite[the statement after (4.22)]{WaldingerLCS}}]\label{numbergens4twoedges}
Let $\sK$be the simplicial complex on four points
deﬁned as the disjoint union of two segments; i.e., $\racg_\sK = (\mathbb Z_2\langle g_1 \rangle \oplus \mathbb Z_2\langle g_2 \rangle) \ast (\mathbb Z_2 \langle g_3 \rangle \oplus \mathbb Z_2\langle g_4 \rangle)$. Then $L^4(\racg_\sK)$ has a minimal set of $15$ generators.
\end{proposition}

\begin{proposition}[{\cite[the statement after (4.22)]{WaldingerLCS}}]\label{numbergens4oneedge}
Let $\sK$be the simplicial complex on four points
deﬁned as the disjoint union of two points and a segment; i.e., $$\racg_\sK = (\mathbb Z_2\langle g_1 \rangle \oplus \mathbb Z_2\langle g_2 \rangle) \ast \mathbb Z_2 \langle g_3 \rangle \ast \mathbb Z_2\langle g_4 \rangle.$$ Then $L^4(\racg_\sK)$ has a minimal set of $23$ generators.
\end{proposition}

\begin{proposition}[{\cite[the statement after (4.22)]{WaldingerLCS}}]\label{numbergens4}
Let $\sK$ be the discrete simplicial complex on four points; i.e., $\racg_\sK = \mathbb Z_2\langle g_1 \rangle \ast \mathbb Z_2\langle g_2 \rangle \ast \mathbb Z_2\langle g_3 \rangle \ast \mathbb Z_2\langle g_4 \rangle$. Then $L^4(\racg_\sK)$ has a minimal generating set of $32$ generators.
\end{proposition}

\begin{theorem}[{\cite[Theorem 4.5]{ve3}}]\label{LRCK}
Let $\sK$ be a simplicial complex on $[m]$, let $\racg_\sK$ be the right-angled Coxeter group corresponding to $\sK$, and let $L(\racg_\sK)$ be the associated Lie algebra. Then
\begin{itemize}
\item[(a)] $L^1(\racg_\sK)$ has the basis $\mu_1, \ldots, \mu_m$;
\item[(b)] $L^2(\racg_\sK)$ has the basis consisting of the commutators $[\mu_i, \mu_j]$, where $i < j$ and ${\{i, j\} \notin \sK}$;
\item[(c)] $L^3(\racg_\sK)$ has the basis consisting of
\begin{itemize}
\item the commutators $[\mu_i, \mu_j, \mu_j]$ with $i < j$ and $\{i, j\} \notin \sK$;
\item the commutators $[\mu_i, \mu_j, \mu_k]$, where $i < j > k, i \neq k$ and $i$ is the least vertex in a connected component of $\sK_{\{i,j,k\}}$ not containing $j$.
\end{itemize}
\end{itemize}
\end{theorem}

We also have the following relations.

\begin{lemma}[{\cite[Lemma 3.3]{ve5}}]\label{comtogPV}
For any group $G$ and any $a, b, c \in G$,
\begin{align}
(a,(b,c))&=(a,c)(c,(b,a))(a,b)(c,b)(b,(a,c))(c,a)(b,a)(b,c),\\
((a,b),c)&=(b,a)(c,a)(c,b)((c,b),a)(a,b)(a,c)((a,c),b)(b,c).
\end{align}
\end{lemma}

\begin{corollary}[{\cite[Corollary 3.8]{ve5}}]\label{cor_freeproduct3}
Let $\sK$ be a discrete simplicial complex on three points; i.e., $$\racg_\sK = \mathbb Z_2\langle g_1 \rangle \ast \mathbb Z_2\langle g_2 \rangle \ast \mathbb Z_2\langle g_3 \rangle.$$ Then $L^4(\racg_\sK)\cong\mathbb Z_2^8$ has the following minimal collection of generators:
$$
[\mu_j, \mu_i, \mu_i, \mu_i], [\mu_k, \mu_i, \mu_i, \mu_i], [\mu_k, \mu_j, \mu_j, \mu_i], [\mu_k, \mu_i, \mu_j, \mu_i],$$
$$[\mu_k, \mu_i, \mu_i, \mu_j],
[\mu_k, \mu_j, \mu_j, \mu_j], [\mu_k, \mu_j, \mu_i, \mu_j], [\mu_k, \mu_i, \mu_j, \mu_k],
$$
where $\mu_i = [g_i]$ is the conjugacy class corresponding to $g_i$ and $(i,j,k)$ is an arbitrary permutation of $(1,2,3)$.
\end{corollary}

\begin{theorem}[{\cite[Theorem 3.9]{ve5}}]\label{commcox3}
Let $\sK$ be a simplicial complex on the set $[3]=\{1,2,3\}$.
\begin{enumerate}[label=(\alph*)]
\item If $\sK$ contains all edges $\{1, 2\}, \{1, 3\}, \{2, 3\}$, then $L^4(\racg_\sK) = \{e\}$ and there are no generators.
\item If $\sK$ contains only two edges $\{i, k\}, \{j, k\}$, where $(i,j,k)$ is a permutation of $(1,2,3)$ and $i<j$, then $L^4(\racg_\sK) \cong \mathbb Z_2$ is minimally generated by the element $[\mu_i, \mu_j, \mu_i, \mu_i]$;
\item If $\sK$ contains only one edge $\{i, j\}$, where $i, j \in \{1, 2, 3\}, i < j$, then $L^4(\racg_\sK) \cong \mathbb Z_2^4$ is minimally generated by the four elements
$$
[\mu_i, \mu_k, \mu_i, \mu_i],~[\mu_k, \mu_j, \mu_k, \mu_k],~[\mu_k, \mu_j, \mu_k, \mu_i],~[\mu_k, \mu_j, \mu_i, \mu_k],
$$
where $k \neq i, k \neq j$.
\item If $\sK$ has no edges, then $L^4(\racg_\sK) \cong \mathbb Z_2^8$ is minimally generated by the eight elements
$$
[\mu_j, \mu_i, \mu_i, \mu_i],~[\mu_k, \mu_i, \mu_i, \mu_i],~[\mu_k, \mu_j, \mu_j, \mu_i],~ [\mu_k, \mu_i, \mu_j, \mu_i],$$
$$[\mu_k, \mu_i, \mu_i, \mu_j],~
[\mu_k, \mu_j, \mu_j, \mu_j],~[\mu_k, \mu_j, \mu_i, \mu_j],~[\mu_k, \mu_i, \mu_j, \mu_k],
$$
where $\mu_i = [g_i]$ is the conjugacy class of $g_i$  and $(i,j,k)$ is an arbitrary permutation of $(1,2,3)$.
\end{enumerate}
\end{theorem}

\begin{theorem}[{\cite[Theorem 3.11]{ve5}}]\label{theorem_freeproduct4}
Let $\sK$ be the discrete simplicial complex on four points; i.e., $$\racg_\sK = \mathbb Z_2\langle g_1 \rangle \ast \mathbb Z_2\langle g_2 \rangle \ast \mathbb Z_2\langle g_3 \rangle \ast \mathbb Z_2\langle g_4 \rangle.$$ Then $L^4(\racg_\sK)\cong\mathbb Z_2^{32}$ has a minimal generating set $\overline{A_i}\cup \overline{A_j}\cup \overline{A_k}\cup \overline{A_l}\cup \overline{B}$, where
\begin{align*}
\overline{A_i}= \{&[\mu_k, \mu_j, \mu_j, \mu_i], [\mu_k, \mu_i, \mu_j, \mu_i], [\mu_k, \mu_i, \mu_i, \mu_j], [\mu_k, \mu_j, \mu_i, \mu_j], [\mu_k, \mu_i, \mu_j, \mu_k] \},\\
\overline{A_j}= \{&[\mu_j, \mu_i, \mu_i, \mu_i], [\mu_l, \mu_j, \mu_j, \mu_i], [\mu_l, \mu_i, \mu_j, \mu_i], [\mu_l, \mu_i, \mu_i, \mu_j], [\mu_l, \mu_j, \mu_i, \mu_j],\\
&[\mu_l, \mu_i, \mu_j, \mu_l] \},\\
\overline{A_k}= \{&[\mu_k, \mu_i, \mu_i, \mu_i], [\mu_l, \mu_i, \mu_i, \mu_i], [\mu_l, \mu_k, \mu_k, \mu_i], [\mu_l, \mu_i, \mu_k, \mu_i], [\mu_l, \mu_i, \mu_i, \mu_k],\\
&[\mu_l, \mu_k, \mu_i, \mu_k], [\mu_l, \mu_i, \mu_k, \mu_l] \},\\
\overline{A_l}= \{&[\mu_k, \mu_j, \mu_j, \mu_j], [\mu_l, \mu_j, \mu_j, \mu_j], [\mu_l, \mu_k, \mu_k, \mu_j], [\mu_l, \mu_j, \mu_k, \mu_j], [\mu_l, \mu_j, \mu_j, \mu_k],\\
&[\mu_l, \mu_k, \mu_k, \mu_k], [\mu_l, \mu_k, \mu_j, \mu_k], [\mu_l, \mu_j, \mu_k, \mu_l] \},\\
\overline{B}= \{&[\mu_j, \mu_l, \mu_k, \mu_i], [\mu_i, \mu_l, \mu_k, \mu_j], [\mu_i, \mu_l, \mu_j, \mu_k], [\mu_j, \mu_l, \mu_i, \mu_k], [\mu_k, \mu_l, \mu_i, \mu_j],\\
&[\mu_k, \mu_l, \mu_j, \mu_i] \},
\end{align*}
$\mu_i = [g_i]$ is the conjugacy class of $g_i$, and $(i,j,k,l)$ is an arbitrary permutation of $(1,2,3,4)$.
\end{theorem}

\section{On the Magnus mapping}

Let $\sK$ be a simplicial complex on $[m]$, and let $V$ be an $m$-dimensional vector space over $\Z_2$. By $$T_{\Z_2}(\sK^0) = \bigoplus_{k=0}^\infty T^k V$$ we denote the tensor algebra with m generators over $\Z_2$. Let $$\widehat{T}_{\Z_2}(\sK^0) = \prod_{k=0}^\infty T^k V$$ be the completed tensor algebra
\begin{definition}
Set
$$
U^{\infty}_{\RCK} = \widehat{T}_{\Z_2}(\K^0) / \langle v_i^2 = 0, \forall i \in \K^0; \ \ v_i v_j + v_i v_j = 0, \{i,j\} \in \K^1 \rangle,
$$
and introduce the Lie bracket by the rule $[a, b] = ab - ba$. The Lie algebra $U^{\infty}_{\RCK}$ is called the \emph{Magnus algebra}.
\end{definition}

Let $U^{\infty, i}_{\RCK}$ consist of all monomials of length i; i.e., $$U^{\infty}_{\RCK} = \prod_{i=0}^{\infty} U^{\infty, i}_{\RCK}.$$ We introduce the associative algebra consisting of elements of finite length in the Magnus algebra,
$$
U_{\RCK} = \bigoplus_{i=0}^{\infty} U^{\infty, i}_{\RCK} = T_{\Z_2}(\K^0) / \langle v_i^2 = 0, \forall i \in \K^0; \ \ v_i v_j + v_j v_i = 0, \{i,j\} \in \K^1 \rangle.
$$

We denote the ideal of monomials of length $\geq k$ in $U_{\RCK}$ and $U^\infty_{\RCK}$ by $$U^{}_{\RCK, k} = \bigoplus\limits_{i \geqslant k}  U^{\infty, i}_{\RCK},\quad U^{\infty}_{\RCK, k} =  \prod\limits_{i \geqslant k}  U^{\infty, i}_{\RCK}.$$

For $a \in U_{RC_\K}^{\infty}$, by $a^i \in U_{RC_\K}^{\infty, i}$ we denote the $i$th graded component of $a$. 

The following theorem holds, which is similar to \cite[Theorem 5.6, Lemma 5.3]{Ma-Car-Sol} for the case of free groups.
 
\begin{theorem}\label{mugomomorph_tem}
Let $\sK$ be a simplicial complex on $[m]$.
\begin{enumerate}
\item The set $\Gamma\subset U^{\infty}_\RCK$ of all elements with nonzero constant term is a group with respect to multiplication.
\item The elements $a_i = 1 + v_i$ generate a subgroup $M \subset \Gamma$ isomorphic to $\RCK$.
\end{enumerate}
\end{theorem}

\begin{proof}
Note that the algebra $U^{\infty}_{\racg_\sK}$ is closed with respect to multiplication and contains 1, and the inverse of an element $g = 1 + h$, $h \in U^{\infty}_{\RCK, 1}$, is the element $$g^{-1} = 1 + \sum_{q=1}^{\infty}h^q.$$

We will prove the second assertion. Let us verify that the necessarry relations in $\racg_\sK$ are satisfied in $M$. We have
\begin{align*}
(1+v_i)^{-1}&= (1+v_i),\\
(1+ v_i)^{-1}(1 + v_j)^{-1}(1 + v_i)(1 + v_j)&= 1 + ( v_i v_j + v_j v_i)(1+ v_i + v_j + v_i v_j).
\end{align*}
Thus, we have a surjective group homomorphism 
$$\mu: \RCK \rightarrow M \subset U_\RCK$$ defined on the generators by $\mu(g_i) = v_i + 1$.

Let $z \in \RCK = F_{|\sK^0|} / \langle g_i^2=0, g_ig_j=g_jg_i \text{ for } \{i, j\} \in \sK \rangle$. There exist finitely many shortest representatives of $z$, i.e., representatives of minimal length in $F_{|\sK^0|}$. Let $w = g_{i_1}...g_{i_k}$ be a shortest representative of $z$. Then
$$\mu(z) = (1+v_{i_1}) \dots (1+v_{i_k}) = 1 + \dots  + v_{i_1} \dots v_{i_k}.$$
Thus, for $z$ with a shortest representative of length $k$ we have $\mu(z)^i = 0$ for $i > k$, and the notation of the graded component $\mu(z)^k$ coincides with one of the shortest representatives of $z$. Therefore, for two words  $z_1, z_2 \in \RCK$ we see that $\mu(z_1) = \mu(z_2)$ implies the existence of the same shortest representatives of $z_1$ and $z_2$, whence it follows that $z_1=z_2$. This means that the homomorphism $\mu$ is an injection, which proves the theorem.
\end{proof}

The homomorphism $$\mu: \RCK \rightarrow M \subset U_\RCK$$ constructed in Theorem~\ref{mugomomorph_tem}  will be called the \emph{Magnus mapping}.

Set $D_i = \{ a \in \RCK: \mu(a) - 1 \in U_{\RCK, i} \}$.

\begin{proposition}
The sequence $\{ D_i\}_{i=1}^\infty$ is a central ceries for $\RCK$. 
\end{proposition}
\begin{proof}
Consider the multiplicative group $(U^{\infty}_\RCK)^*$ and the multiplicative group $\Gamma$ of elements with constant term $1$. According to \cite[Chap. 2, \textsection 4,5]{Bourbaki}, the filtration $1 + U^{\infty}_{\RCK, i}$ is a central series for $\Gamma$, and also $\Gamma \subset (U^{\infty}_\RCK)^*$. Actually, $\Gamma = (U^{\infty}_\RCK)^*$, which follows from the fact that the product of two elements has zero constant term whenever at least one of factors has zero constant term.  

Further, $M \cong \RCK$ is a subgroup of $\Gamma$ with the filtration $M_i = \mu(D_i) \subset 1 + U^{\infty}_{\RCK, i}$ induced from $\{ D_i\}_{i=1}^\infty$. Thus, on the one hand, 
$$(M_k, M_l) \subset (1 + U^{\infty}_{\RCK, k}, 1 + U^{\infty}_{\RCK, l}) \subset 1 + U^{\infty}_{\RCK, k+l},$$ 
and on the other hand, $(M_k, M_l) \subset M$, whence it follows that $(M_k, M_l) \subset M_{k+l}$. Thus, $\{M_k\}_{i=1}^\infty$ is a central series for $M$, and hence $\{ D_i\}_{i=1}^\infty$ is a central series for $\RCK$.
\end{proof}

\begin{corollary} \label{muk} Let $x \in \RCK$. If $\mu(x)^k \neq 0$, then $x \notin \gamma_{k+1}(\racg_\sK)$.
\end{corollary}
\begin{proof}
We haev $\gamma_k(\racg_\sK) \subset D_k$, because $\gamma_k(\racg_\sK)$ is the lower central series. Hence for an element $x \in \gamma_k(\racg_\sK)$ (a commutator of length $k$) one has $\mu(x)^i = 0$, $0 < i < k$, as desired.
\end{proof}

\begin{proposition}\label{notgamma5elems}
Given a simplicial complex $\mathcal{K}$ on four vertices $[4] = \{ v_1,  v_2 ,  v_3 , v_4\}$ with edges $\{ \{ v_4 v_1\}, \{ v_4  v_3 \} \}$, one has
$$( (v_2, v_4) ,  (v_1, v_3) ),~( v_2  v_4  v_1  v_3 ),~( v_2  v_4  v_3  v_1 ) \notin \gamma_5.$$ 
\end{proposition}
\begin{proof}
Consider the commutator $( (v_2, v_4) ,  (v_1, v_3) )$. We have
\begin{align*}
\mu(( (v_2, v_4) ,  (v_1, v_3) )) =  1&+ 
v_{2} v_{4} v_{1} v_{3} + v_{2} v_{4} v_{3} v_{1} + v_{4} v_{2} v_{1} v_{3} + v_{4} v_{2} v_{3} v_{1} \\ 
&+ v_{1} v_{3} v_{2} v_{4} + v_{1} v_{3} v_{4} v_{2} + v_{3} v_{1} v_{2} v_{4} + v_{3} v_{1} v_{4} v_{2} + U_{\RCK, 5}.
\end{align*}
Let us show that $$\mu(( (v_2, v_4) ,  (v_1, v_3) ))^4 \neq 0,$$ whence, by Corollary \ref{muk}, we obtain $$( (v_2, v_4) ,  (v_1, v_3) ) \notin \gamma_5(\racg_\sK).$$ The product of four distinct generators cannot be zero, and hence each term in $\mu(( (v_2, v_4) ,  (v_1, v_3) ))^4$ is not zero by itself. Further, consider the term $v_{4} v_{2} v_{1} v_{3} $. Since any two neighboring generators in this
representation do not commute, it follows that the monomial $v_{4} v_{2} v_{1} v_{3} $ can only be written in this unique
way. However, according to the relations in the algebra, this term can only be cancelled by the term
equal to it.

Further, since the commutators $( v_2  v_4  v_1  v_3 )$ and $( v_2  v_4  v_3  v_1 )$ differ by the change $v_1 \leftrightarrow v_3$, which preserves the relations in the algebra, it suffices to prove the assertion for one of these commutators.
We have
\begin{multline*}
\mu(( v_2, v_4, v_3, v_1 )) = 1 + 
v_{1} v_{3} v_{2} v_{4} + v_{1} v_{3} v_{4} v_{2} + v_{1} v_{2} v_{4} v_{3} + v_{1} v_{4} v_{2} v_{3} + \\
+ v_{3} v_{2} v_{4} v_{1} + v_{3} v_{4} v_{2} v_{1} + v_{2} v_{4} v_{3} v_{1} + v_{4} v_{2} v_{3} v_{1} + U_{\RCK, 5}.
\end{multline*}
As in the preceding argument, we see that the monomial $ v_{1} v_{3} v_{2} v_{4} \neq 0$ can only be written in this unique way, which implies that $\mu(( v_2, v_4, v_3, v_1 ))^4 \neq 0$, and hence $( ( v_2, v_4, v_3, v_1 ) ) \notin \gamma_5$.
\end{proof}

\begin{proposition}\label{notgamma5elems2}
	Given a simplicial complex $\mathcal{K}$ on four vertices $[4] = \{ v_1,  v_2 ,  v_3 , v_4\}$ with edges $\{ \{ v_{3}, v_{1}\}, \{ v_{1},  v_{4} \}, \{v_{4}, v_{2} \} \}$, one has $( v_{3}, v_{4}, v_{2}, v_{1} ) \notin \gamma_5$. 
\end{proposition}
\begin{proof}
	We have $( v_{3}, v_{4}, v_{2}, v_{1} ) = ( (v_{3}, v_{4}), (v_{1}, v_{2}) ) \operatorname{mod} \gamma_5$. Therefore, it suffices to perform the
calculations for the commutator $( (v_{1}, v_{2}) ,  (v_{3}, v_{4}) )$. We have
	\begin{align*}
		\mu(( (v_{1}, v_{2}) ,  (v_{3}, v_{4}) )) =  1 +& 
		v_{{1}} v_{{2}} v_{{3}} v_{{4}} + v_{{1}} v_{{2}} v_{{4}} v_{{3}} + v_{{2}} v_{{1}} v_{{3}} v_{{4}} + v_{{2}} v_{{1}} v_{{4}} v_{{3}} \\ 
		+&v_{{3}} v_{{4}} v_{{1}} v_{{2}} + v_{{3}} v_{{4}} v_{{2}} v_{{1}} + v_{{4}} v_{{3}} v_{{1}} v_{{2}} + v_{{4}} v_{{3}} v_{{2}} v_{{1}} + U_{\RCK, 5}.
	\end{align*}
	Repeating the argument in the preceding proposition, it suffices to note that the monomial $v_{{1}} v_{{2}} v_{{3}} v_{{4}} $ 
	can
only be written in this unique way and hence cannot be cancelled.	
\end{proof}

\section{Main results}

\begin{proposition}\label{coxeter4pointsgens}
Let $\sK$ be a simplicial complex on the set $[4]=\{1,2,3,4\}$.
\begin{enumerate}[label=(\alph*)]
\item If $\sK$ contains all edges $\{1, 2\}, \{1, 3\}, \{1, 4\}, \{2, 3\}, \{2, 4\}, \{3, 4\}$, then\\$L^4(\racg_\sK) = \{e\}$ and there are no generators.
\item If $\sK$ contains only five edges and the only missing edge is $\{i, j\}, i < j$, then $L^4(\racg_\sK) \cong \mathbb Z_2$ is minimally generated by the commutator $[\mu_j, \mu_i, \mu_i, \mu_i]$.
\item If $\sK$ contains only four edges and the missing edges are
\begin{itemize}
\item $\{i, k\}, \{j, l\}$, $k < i, l < j$, then $L^4(\racg_\sK)\cong\mathbb Z_2\oplus\mathbb Z_2$ is minimally generated by the two commutators $[\mu_i, \mu_k, \mu_k, \mu_k], [\mu_j, \mu_l, \mu_l, \mu_l]$;
\item $\{i, k\}, \{k, l\}$, $i < k < l$, then $L^4(\racg_\sK)\cong\mathbb Z_2^4$ is minimally generated by the four commutators $$[\mu_k, \mu_i, \mu_i, \mu_i],[\mu_l, \mu_k, \mu_k, \mu_k],[\mu_k, \mu_l, \mu_k, \mu_i],[\mu_k, \mu_l, \mu_i, \mu_k].$$
\end{itemize}
\item If $\sK$ contains only three edges
\begin{itemize}
\item $\{l, j\}, \{l, k\}$, $\{l, i\}$, $i < j < k$, then $L^4(\racg_\sK)\cong\mathbb Z_2^8$ is minimally generated by the eight commutators 
\begin{align*}
&[\mu_j, \mu_i, \mu_i, \mu_i],[\mu_k, \mu_i, \mu_i, \mu_i],[\mu_k, \mu_j, \mu_j, \mu_i],[\mu_k, \mu_i, \mu_j, \mu_i],
\\&
[\mu_k, \mu_i, \mu_i, \mu_j],
[\mu_k, \mu_j, \mu_j, \mu_j],
[\mu_k, \mu_j, \mu_i, \mu_j],
[\mu_k, \mu_i, \mu_j, \mu_k].
\end{align*}
\item $\{j, k\}, \{j, l\}, \{k, l\}$, $j < k < l$, then $L^4(\racg_\sK)\cong\mathbb Z_2^{10}$ is minimally generated by the ten commutators
\begin{align*}
&[\mu_i, \mu_k, \mu_i, \mu_j],[\mu_i, \mu_k, \mu_j, \mu_i],[\mu_i, \mu_j, \mu_j, \mu_j],[\mu_i, \mu_l, \mu_i, \mu_j],[\mu_i, \mu_l, \mu_j, \mu_i],\\
&[\mu_i, \mu_k, \mu_k, \mu_k],[\mu_l, \mu_i, \mu_i, \mu_i],[\mu_i, \mu_l, \mu_i, \mu_k],[\mu_i, \mu_l, \mu_k, \mu_i],[\mu_i, \mu_l, \mu_j, \mu_k].
\end{align*}
\item $\{i, j\}, \{j, k\}, \{k, l\}$, then $L^4(\racg_\sK)\cong\mathbb Z_2^8$ is minimally generated by the eight commutators 
\begin{align*}
&[\mu_k, \mu_i, \mu_k, \mu_k],[\mu_i, \mu_l, \mu_i, \mu_k],[\mu_i, \mu_l, \mu_k, \mu_i],[\mu_i, \mu_l, \mu_i, \mu_i],\\
&[\mu_l, \mu_j, \mu_l, \mu_l],[\mu_l, \mu_j, \mu_l, \mu_i],[\mu_l, \mu_j, \mu_i, \mu_l],[\mu_l, \mu_j, \mu_i, \mu_k].
\end{align*}
\end{itemize}
\item if $\sK$ contains only two edges
\begin{itemize}
\item $\{i, j\}$ and $\{i, l\}$, then $L^4(\racg_\sK)\cong\mathbb Z_2^{15}$ is minimally generated by the 15 commutators
\begin{align*}
&[\mu_l, \mu_k, \mu_k, \mu_j],[\mu_l, \mu_j, \mu_k, \mu_j],[\mu_l, \mu_j, \mu_j, \mu_k],[\mu_l, \mu_k, \mu_j, \mu_k],[\mu_l, \mu_j, \mu_k, \mu_l],\\
&[\mu_i, \mu_k, \mu_k, \mu_j],[\mu_i, \mu_k, \mu_j, \mu_k],[\mu_k, \mu_j, \mu_j, \mu_j],[\mu_l, \mu_j, \mu_j, \mu_j],[\mu_l, \mu_k, \mu_l, \mu_l],\\
&[\mu_k, \mu_i, \mu_k, \mu_k],[\mu_k, \mu_i, \mu_k, \mu_l],[\mu_k, \mu_i, \mu_l, \mu_k],[\mu_k, \mu_i, \mu_l, \mu_j],[\mu_k, \mu_i, \mu_j, \mu_l].
\end{align*}
\item $\{i, j\}$ and $\{k, l\}$, then $L^4(\racg_\sK)\cong\mathbb Z_2^{15}$ is minimally generated by the 15 commutators
\begin{align*}
&[\mu_k, \mu_j, \mu_k, \mu_i],
[\mu_k, \mu_j, \mu_i, \mu_k],
[\mu_k, \mu_i, \mu_k, \mu_k],
[\mu_i, \mu_l, \mu_i, \mu_k],
[\mu_i, \mu_l, \mu_k, \mu_i]\\
&[\mu_i, \mu_l, \mu_i, \mu_i],
[\mu_l, \mu_j, \mu_l, \mu_i],
[\mu_l, \mu_j, \mu_i, \mu_l],
[\mu_k, \mu_j, \mu_k, \mu_k],
[\mu_j, \mu_l, \mu_j, \mu_j]\\&
[\mu_j, \mu_l, \mu_j, \mu_k],
[\mu_j, \mu_l, \mu_k, \mu_j],
[\mu_i, \mu_l, \mu_k, \mu_j],
[\mu_j, \mu_l, \mu_i, \mu_k],
[\mu_j, \mu_l, \mu_k, \mu_i].
\end{align*}
\end{itemize}
\item If $\sK$ contains only one edge $\{i, j\}$, then $L^4(\racg_\sK)\cong\mathbb Z_2^{23}$ is minimally generated by the 23 commutators
\begin{align*}
&[\mu_i, \mu_l, \mu_l, \mu_k],~
[\mu_i, \mu_k, \mu_l, \mu_k],~
[\mu_i, \mu_k, \mu_k, \mu_l],~
[\mu_i, \mu_l, \mu_k, \mu_l],~
[\mu_i, \mu_k, \mu_l, \mu_i],\\&
[\mu_l, \mu_k, \mu_k, \mu_k],~
[\mu_j, \mu_l, \mu_l, \mu_k],~
[\mu_j, \mu_k, \mu_l, \mu_k],~
[\mu_j, \mu_k, \mu_k, \mu_l],~
[\mu_j, \mu_l, \mu_k, \mu_l],\\&
[\mu_j, \mu_k, \mu_l, \mu_j],~
[\mu_i, \mu_k, \mu_i, \mu_i],~
[\mu_k, \mu_j, \mu_k, \mu_k],~
[\mu_k, \mu_j, \mu_k, \mu_i],~
[\mu_k, \mu_j, \mu_i, \mu_k],\\&
[\mu_i, \mu_l, \mu_i, \mu_i],~
[\mu_l, \mu_j, \mu_l, \mu_l],~
[\mu_l, \mu_j, \mu_l, \mu_i],~
[\mu_l, \mu_j, \mu_i, \mu_l],\\&
[\mu_l, \mu_j, \mu_i, \mu_k],~
[\mu_k, \mu_j, \mu_i, \mu_l],~
[\mu_k, \mu_j, \mu_l, \mu_i],~
[\mu_l, \mu_j, \mu_k, \mu_i].
\end{align*}
\end{enumerate}
\end{proposition}

\begin{proof}
It follows from Theorem~\ref{theorem_freeproduct4} that in all the cases the union of the sets
\begin{align}
\nonumber
\overline{A_1} = \{&[\mu_3, \mu_2, \mu_2, \mu_1], [\mu_3, \mu_1, \mu_2, \mu_1], [\mu_3, \mu_1, \mu_1, \mu_2],\\
\nonumber
&[\mu_3, \mu_2, \mu_1, \mu_2], [\mu_3, \mu_1, \mu_2, \mu_3] \},\\
\nonumber
\overline{A_2} = \{&[\mu_2, \mu_1, \mu_1, \mu_1], [\mu_4, \mu_2, \mu_2, \mu_1], [\mu_4, \mu_1, \mu_2, \mu_1],\\&
\nonumber
[\mu_4, \mu_1, \mu_1, \mu_2], [\mu_4, \mu_2, \mu_1, \mu_2], [\mu_4, \mu_1, \mu_2, \mu_4] \},\\
\label{temp_setcomm4}
\overline{A_3} = \{&[\mu_3, \mu_1, \mu_1, \mu_1], [\mu_4, \mu_1, \mu_1, \mu_1], [\mu_4, \mu_3, \mu_3, \mu_1], [\mu_4, \mu_1, \mu_3, \mu_1],\\
\nonumber
&[\mu_4, \mu_1, \mu_1, \mu_3], [\mu_4, \mu_3, \mu_1, \mu_3], [\mu_4, \mu_1, \mu_3, \mu_4] \},\\
\nonumber
\overline{A_4} = \{&[\mu_3, \mu_2, \mu_2, \mu_2], [\mu_4, \mu_2, \mu_2, \mu_2], [\mu_4, \mu_3, \mu_3, \mu_2], [\mu_4, \mu_2, \mu_3, \mu_2],\\
\nonumber
&[\mu_4, \mu_2, \mu_2, \mu_3], [\mu_4, \mu_3, \mu_3, \mu_3], [\mu_4, \mu_3, \mu_2, \mu_3], [\mu_4, \mu_2, \mu_3, \mu_4] \},\\
\nonumber
\overline{B} = \{&[\mu_2, \mu_4, \mu_3, \mu_1], [\mu_1, \mu_4, \mu_3, \mu_2], [\mu_1, \mu_4, \mu_2, \mu_3], [\mu_2, \mu_4, \mu_1, \mu_3],\\
\nonumber
&[\mu_3, \mu_4, \mu_1, \mu_2], [\mu_3, \mu_4, \mu_2, \mu_1] \},
\end{align}

generates $L^4(\racg_\sK)$ (not necessarily minimally and, in fact, never minimally).

In item (a), in the presence of all edges, all elements in the union become zero, which implies the
desired assertion.

In item (b), consider the case in which the edge $\{1, 2\}$ is missing (i.e., $i = 1$ and $j = 2$). Then all the
generators except for $[\mu_2, \mu_1, \mu_1, \mu_1]$ are zero, which implies that the assertion is true in this case. Since
one can always arbitrarily change the numbering of vertices, it follows that for any missing edge one has $L^4(\racg_\sK) \cong \mathbb Z_2$, and one can take any nonzero commutator from this union, which implies the desired
assertion.

In item (c), two cases of the simplicial complex (graph) are possible up to an isomorphism:
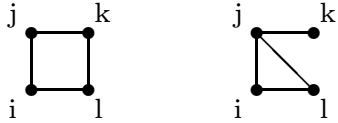
\begin{center}
\scalebox{1.5}{
\begin{picture}(10,10)
\put(0,2){\circle*{1}}
\put(5,2){\circle*{1}}
\put(0,7){\circle*{1}}
\put(5,7){\circle*{1}}
\put(0,2){\line(1,0){5}}
\put(0,2){\line(0,1){5}}
\put(5,7){\line(-1,0){5}}
\put(5,7){\line(0,-1){5}}
\put(-2,-0.5){\scriptsize i}
\put(-2,8){\scriptsize j}
\put(5.5,8){\scriptsize k}
\put(5.5,-0.5){\scriptsize l}
\end{picture}}\;\;\;\;\;\;\;\;\;\;\;\;
\scalebox{1.5}{
\begin{picture}(10,10)
\put(0,2){\circle*{1}}
\put(5,2){\circle*{1}}
\put(0,7){\circle*{1}}
\put(5,7){\circle*{1}}
\put(0,2){\line(1,0){5}}
\put(0,2){\line(0,1){5}}
\put(0,7){\line(1,-1){5}}
\put(5,7){\line(-1,0){5}}
\put(-2,-0.5){\scriptsize i}
\put(-2,8){\scriptsize j}
\put(5.5,8){\scriptsize k}
\put(5.5,-0.5){\scriptsize l}
\end{picture}}
\end{center}
In the first case, the simplicial complex $\sK$ is the join of two discrete simplicial complexes on two points ($\{i, k\}$ and $\{j, l\}$), and therefore $\racg_\sK \cong (\mathbb Z_2 \ast \mathbb Z_2) \oplus (\mathbb Z_2 \ast \mathbb Z_2)$ by~\ref{coxfund} and~\cite[Propositions 4.1.3 and 4.2.5]{bu-pa15}. The lower central seres of the group $\mathbb{Z}_2 \ast \mathbb{Z}_2$ has the following form: $\gamma_1 (\mathbb{Z}_2 \ast \mathbb{Z}_2) = \mathbb{Z}_2 \ast \mathbb{Z}_2$, and for $k \ge 2$ $\gamma_k (\mathbb{Z}_2 \ast \mathbb{Z}_2) \cong \mathbb{Z}$ is the infinite cyclic group generated by the commutator $(g_1, g_2, g_1, \ldots, g_1)$ of length $k$. From Proposition~\ref{kv},we obtain $$\gamma_k(\mathbb{Z}_2 \ast \mathbb{Z}_2) / \gamma_{k+1}(\mathbb{Z}_2 \ast \mathbb{Z}_2) = \mathbb{Z}_2$$ for all $k > 1$, while $\gamma_1(\mathbb{Z}_2 \ast \mathbb{Z}_2) / \gamma_2(\mathbb{Z}_2 \ast \mathbb{Z}_2) = \mathbb{Z}_2 \oplus \mathbb{Z}_2$. It follows that $$\gamma_4(\racg_\sK) / \gamma_5(\racg_\sK) = L^4(\racg_\sK) \cong \mathbb Z_2 \oplus \mathbb Z_2$$ is minimally generated by two commutators $[\mu_i, \mu_k, \mu_k, \mu_k], [\mu_j, \mu_l, \mu_l, \mu_l]$.

In the second case, we add the simplex $\{i, j, l\}$ to the simplicial comples $\sK$, which does not change the graph $\sK^1$ (note that the group $\racg_\sK$ depends only on $\sK^1$). Then the resulting simplicial complex $s\K'$ satisfies $\sK' \cong \{j\} \ast \sK_{\{i, l, k\}}$, whence $\racg_\sK \cong \mathbb Z_2 \oplus \racg_{\sK_{\{i, l, k\}}}$ by Theorem~\ref{coxfund} and~\cite[Propositions 4.1.3 and 4.2.5]{bu-pa15}. Since $\gamma_4(\mathbb Z_2) \cong \{e\}$, it follows that $L^4(\racg_\sK)\cong L^4(\racg_{\sK_{\{i, l, k\}}})\cong\mathbb Z_2^4$ is minimally generated by the four
commutators $[\mu_i, \mu_k, \mu_i, \mu_i], [\mu_k, \mu_l, \mu_k, \mu_k], [\mu_k, \mu_l, \mu_k, \mu_i], [\mu_k, \mu_l, \mu_i, \mu_k]$ by Theorem~\ref{commcox3}.

In item (d), three cases of the simplicial complex (graph) are possible up to an isomorphism:
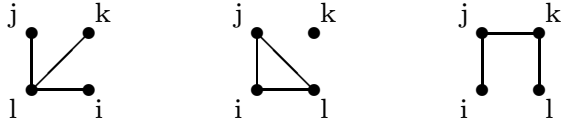
\begin{center}
\scalebox{1.5}{
\begin{picture}(10,10)
\put(0,2){\circle*{1}}
\put(5,2){\circle*{1}}
\put(0,7){\circle*{1}}
\put(5,7){\circle*{1}}
\put(0,2){\line(1,0){5}}
\put(0,2){\line(0,1){5}}
\put(0,2){\line(1,1){5}}
\put(-2,-0.5){\scriptsize l}
\put(-2,8){\scriptsize j}
\put(5.5,8){\scriptsize k}
\put(5.5,-0.5){\scriptsize i}
\end{picture}}\;\;\;\;\;\;\;\;\;\;\;\;
\scalebox{1.5}{
\begin{picture}(10,10)
\put(0,2){\circle*{1}}
\put(5,2){\circle*{1}}
\put(0,7){\circle*{1}}
\put(5,7){\circle*{1}}
\put(0,2){\line(1,0){5}}
\put(0,2){\line(0,1){5}}
\put(5,2){\line(-1,1){5}}
\put(-2,-0.5){\scriptsize i}
\put(-2,8){\scriptsize j}
\put(5.5,8){\scriptsize k}
\put(5.5,-0.5){\scriptsize l}
\end{picture}}\;\;\;\;\;\;\;\;\;\;\;\;
\scalebox{1.5}{
\begin{picture}(10,10)
\put(0,2){\circle*{1}}
\put(5,2){\circle*{1}}
\put(0,7){\circle*{1}}
\put(5,7){\circle*{1}}
\put(0,2){\line(0,1){5}}
\put(5,7){\line(-1,0){5}}
\put(5,7){\line(0,-1){5}}
\put(-2,-0.5){\scriptsize i}
\put(-2,8){\scriptsize j}
\put(5.5,8){\scriptsize k}
\put(5.5,-0.5){\scriptsize l}
\end{picture}}
\end{center}
In the first case, the simplicial complex $\sK$ is the joint of the point $\{l\}$ and the discrete simplicial complex on the three points $\{i, j, k\}$,  whence, as in the above argument, $L^4(\racg_\sK)\cong L^4(\racg_{\sK_{\{i, j, k\}}})\cong\mathbb Z_2^8$ by Corollary~\ref{cor_freeproduct3}, and $L^4(\racg_\sK)$ is minimally generated by the commutators
\begin{align*}
&[\mu_j, \mu_i, \mu_i, \mu_i],~[\mu_k, \mu_i, \mu_i, \mu_i],~[\mu_k, \mu_j, \mu_j, \mu_i],~[\mu_k, \mu_i, \mu_j, \mu_i],\\
&
[\mu_k, \mu_i, \mu_i, \mu_j],~[\mu_k, \mu_j, \mu_j, \mu_j], ~[\mu_k, \mu_j, \mu_i, \mu_j],~[\mu_k, \mu_i, \mu_j, \mu_k],
\end{align*}
which implies the desired assertion.

In the second case, first, we take $i = 2, j = 4, k = 3, l = 1$. Then the first four commutators in the subset $\overline{B}$ of the set~\eqref{temp_setcomm4} are zero, the remaining commutators being $[\mu_3, \mu_4, \mu_1, \mu_2]$ and $[\mu_3, \mu_4, \mu_2, \mu_1]$. Further, $(g_3, g_4, g_1, g_2)\equiv((g_2, g_1), (g_3, g_4))(g_3, g_4, g_2, g_1)\mod\gamma_5(\racg_\sK) = (g_3, g_4, g_2, g_1)$ (beacuse $(g_1, g_2) = e$);hence we can retain any of these two commutators. We take $[\mu_3, \mu_4, \mu_2, \mu_1]$. Owing to the symmetry of the indices, we see that in any case there remains only one explicitly nonzero commutator with four pairwise distinct indices. The fact that it is not zero will be proved below.

\vspace{0.1cm}
Now we take $i = 2, j = 3, k = 1, l = 4$:
\begin{center}
\scalebox{1.5}{
\begin{picture}(6,7)
\put(0,2){\circle*{1}}
\put(5,2){\circle*{1}}
\put(0,7){\circle*{1}}
\put(5,7){\circle*{1}}
\put(0,2){\line(1,0){5}}
\put(0,2){\line(0,1){5}}
\put(5,2){\line(-1,1){5}}
\put(-2,-0.5){\scriptsize 2}
\put(-2,8){\scriptsize 3}
\put(5.5,8){\scriptsize 1}
\put(5.5,-0.5){\scriptsize 4}
\end{picture}}\end{center}
 Then the subset $\overline{A_4}$ of the set~\eqref{temp_setcomm4} of commutators becomes empty (all commutators in it are zero). In the set $\overline{B}$, we retain only one commutator $[\mu_1, \mu_4, \mu_2, \mu_3]$. We also remove all commutators that are
explicitly zero and gather all commutators with indices in the set $\{1, 2, 3\}$ into a separate set $A'_1$. The following set of commutators remains:
\begin{align*}
A_1'&= \{ [\mu_3, \mu_1, \mu_2, \mu_1], [\mu_3, \mu_1, \mu_1, \mu_2], [\mu_3, \mu_1, \mu_2, \mu_3], [\mu_2, \mu_1, \mu_1, \mu_1], [\mu_3, \mu_1, \mu_1, \mu_1] \},\\
\overline{A_2}&= \{ [\mu_4, \mu_1, \mu_2, \mu_1], [\mu_4, \mu_1, \mu_1, \mu_2], [\mu_4, \mu_1, \mu_2, \mu_4] \},
\\
\overline{A_3}&= \{ [\mu_4, \mu_1, \mu_1, \mu_1], [\mu_4, \mu_1, \mu_3, \mu_1], [\mu_4, \mu_1, \mu_1, \mu_3], [\mu_4, \mu_1, \mu_3, \mu_4] \},
\\\overline{B}&= \{ [\mu_1, \mu_4, \mu_2, \mu_3] \}.
\end{align*}
By Theorem~\ref{commcox3}, all commutators in the set $A_1'$ can be replaced by the four commutators $[\mu_2, \mu_1, \mu_2, \mu_2], [\mu_1, \mu_3, \mu_1, \mu_1], [\mu_1, \mu_3, \mu_1, \mu_2], [\mu_1, \mu_3, \mu_2, \mu_1]$. After this, we gather all commutators with indices
in $\{1, 2, 4\}$ into a set $A_2'$,
\begin{align*}
A_1'&= \{ [\mu_1, \mu_3, \mu_1, \mu_1], [\mu_1, \mu_3, \mu_1, \mu_2], [\mu_1, \mu_3, \mu_2, \mu_1] \},
\\A_2'&= \{ [\mu_4, \mu_1, \mu_2, \mu_1], [\mu_2, \mu_1, \mu_2, \mu_2], [\mu_4, \mu_1, \mu_1, \mu_2], [\mu_4, \mu_1, \mu_2, \mu_4], [\mu_4, \mu_1, \mu_1, \mu_1] \},
\\\overline{A_3}&= \{ [\mu_4, \mu_1, \mu_3, \mu_1], [\mu_4, \mu_1, \mu_1, \mu_3], [\mu_4, \mu_1, \mu_3, \mu_4] \},
\\\overline{B}&= \{ [\mu_1, \mu_4, \mu_2, \mu_3] \}.
\end{align*}
In a similar way, we see that all commutators in the set $A_2'$ can be replaced by the four commutators $[\mu_2, \mu_1, \mu_2, \mu_2], [\mu_1, \mu_4, \mu_1, \mu_1], [\mu_1, \mu_4, \mu_1, \mu_2], [\mu_1, \mu_4, \mu_2, \mu_1]$. After this, we gather all commutators with indices in $\{1, 3, 4\}$ into a set $A_3'$,
\begin{align*}
A_1'&= \{ [\mu_1, \mu_3, \mu_1, \mu_2], [\mu_1, \mu_3, \mu_2, \mu_1] \},
\\A_2'&= \{ [\mu_2, \mu_1, \mu_2, \mu_2], [\mu_1, \mu_4, \mu_1, \mu_2], [\mu_1, \mu_4, \mu_2, \mu_1] \},
\\A_3'&= \{ [\mu_1, \mu_3, \mu_1, \mu_1], [\mu_4, \mu_1, \mu_3, \mu_1], [\mu_4, \mu_1, \mu_1, \mu_3], [\mu_4, \mu_1, \mu_3, \mu_4], [\mu_1, \mu_4, \mu_1, \mu_1] \},
\\\overline{B}&= \{ [\mu_1, \mu_4, \mu_2, \mu_3] \}.
\end{align*}
In a similar way, we see that all commutators in the set $A_3'$ can be replaced by the four commutators $[\mu_3, \mu_1, \mu_3, \mu_3], [\mu_1, \mu_4, \mu_1, \mu_1], [\mu_1, \mu_4, \mu_1, \mu_3], [\mu_1, \mu_4, \mu_3, \mu_1]$. We obtain the following set of commutators:
\begin{align*}
A_1'&= \{ [\mu_1, \mu_3, \mu_1, \mu_2], [\mu_1, \mu_3, \mu_2, \mu_1] \},
\\A_2'&= \{ [\mu_2, \mu_1, \mu_2, \mu_2], [\mu_1, \mu_4, \mu_1, \mu_2], [\mu_1, \mu_4, \mu_2, \mu_1] \},
\\A_3'&= \{ [\mu_3, \mu_1, \mu_3, \mu_3], [\mu_1, \mu_4, \mu_1, \mu_1], [\mu_1, \mu_4, \mu_1, \mu_3], [\mu_1, \mu_4, \mu_3, \mu_1] \},
\\\overline{B}&= \{ [\mu_1, \mu_4, \mu_2, \mu_3] \}.
\end{align*}
This set of commutators is minimal by Proposition~\ref{numbergens4onetriangle}. By the symmetry of the indices, we see that, in
the absence of the edges $\{i, j\}, \{i, k\}, \{i, l\}$ the set
\begin{align*}
A_i'&= \{ [\mu_i, \mu_k, \mu_i, \mu_j], [\mu_i, \mu_k, \mu_j, \mu_i] \},
\\A_j'&= \{ [\mu_i, \mu_j, \mu_j, \mu_j], [\mu_i, \mu_l, \mu_i, \mu_j], [\mu_i, \mu_l, \mu_j, \mu_i] \},
\\A_k'&= \{ [\mu_i, \mu_k, \mu_k, \mu_k], [\mu_l, \mu_i, \mu_i, \mu_i], [\mu_i, \mu_l, \mu_i, \mu_k], [\mu_i, \mu_l, \mu_k, \mu_i] \},
\\
\overline{B}&= \{ [\mu_i, \mu_l, \mu_j, \mu_k] \}
\end{align*}
of commutators minimally generates $L^4(\racg_\sK)\cong\mathbb Z_2^{10}$.

In the third case, we take $i = 2, j = 4, k = 1, l = 3$:
\begin{center}
\scalebox{1.5}{
\begin{picture}(6,9)
\put(0,2){\circle*{1}}
\put(5,2){\circle*{1}}
\put(0,7){\circle*{1}}
\put(5,7){\circle*{1}}
\put(0,2){\line(0,1){5}}
\put(5,7){\line(-1,0){5}}
\put(5,7){\line(0,-1){5}}
\put(-2,-0.5){\scriptsize 2}
\put(-2,8){\scriptsize 4}
\put(5.5,8){\scriptsize 1}
\put(5.5,-0.5){\scriptsize 3}
\end{picture}}
\end{center} Then the first four commutators in the subset $\overline{B}$ of the set~\eqref{temp_setcomm4} are zero, the remaining commutators
being $[\mu_3, \mu_4, \mu_1, \mu_2]$ and $[\mu_3, \mu_4, \mu_2, \mu_1]$. By the Jacobi identity,
$$[\mu_3, \mu_4, \mu_1, \mu_2] = [\mu_4, \mu_1, \mu_3, \mu_2] + [\mu_1, \mu_3, \mu_4, \mu_2] = 0 + 0 = 0.$$ This leaves us with the only explicitly nonzero commutator $[\mu_3, \mu_4, \mu_2, \mu_1]$. In view of the symmetry of
the indices, we see that in any case there is only one explicitly nonzero commutator with four pairwise
distinct indices, namely, $[\mu_l, \mu_j, \mu_i, \mu_k]$. The fact that it is nonzero will be proved below. Further, we
remove the commutators that are explicitly zero from the resulting set of commutators and gather all
commutators with indices in $\{1, 2, 3\}$ into a set $A'_1$. The set of remaining commutators is
\begin{align*}
A_1'&= \{ [\mu_3, \mu_2, \mu_2, \mu_1], [\mu_2, \mu_1, \mu_1, \mu_1], [\mu_3, \mu_2, \mu_1, \mu_2], [\mu_3, \mu_2, \mu_2, \mu_2] \},
\\\overline{A_3}&= \{ [\mu_4, \mu_3, \mu_3, \mu_1], [\mu_4, \mu_3, \mu_1, \mu_3] \},
\\\overline{A_4}&= \{ [\mu_4, \mu_3, \mu_3, \mu_2], [\mu_4, \mu_3, \mu_3, \mu_3], [\mu_4, \mu_3, \mu_2, \mu_3] \},
\\\overline{B}&= \{ [\mu_3, \mu_4, \mu_2, \mu_1] \}.
\end{align*}
By Theorem~\ref{commcox3}, all commutators in the set $A'_1$ can be replaced by the four commutators $[\mu_1, \mu_2, \mu_1, \mu_1], [\mu_2, \mu_3, \mu_2, \mu_2], [\mu_2, \mu_3, \mu_1, \mu_2], [\mu_2, \mu_3, \mu_1, \mu_2]$. There is only one commutator with indices $\{1, 2, 4\}$,
just as claimed in Theorem~\ref{commcox3}. We gather all commutators with indices in $\{1, 3, 4\}$ into a set $A_3'$:
\begin{align*}
A_1'&= \{ [\mu_1, \mu_2, \mu_1, \mu_1], [\mu_2, \mu_3, \mu_2, \mu_2], [\mu_2, \mu_3, \mu_2, \mu_1], [\mu_2, \mu_3, \mu_1, \mu_2] \},
\\A_3'&= \{ [\mu_4, \mu_3, \mu_3, \mu_1], [\mu_4, \mu_3, \mu_1, \mu_3], [\mu_4, \mu_3, \mu_3, \mu_3] \},
\\\overline{A_4}&= \{ [\mu_4, \mu_3, \mu_3, \mu_2], [\mu_4, \mu_3, \mu_2, \mu_3] \},
\\\overline{B}&= \{ [\mu_3, \mu_4, \mu_2, \mu_1] \}.
\end{align*}
In a similar way, we see that all commutators in the set $A_3'$ can be replaced by the single commutator $[\mu_4, \mu_3, \mu_3, \mu_3]$, after which we gather all commutators with indices in $\{2, 3, 4\}$ into a set $A'_4$,
\begin{align*}
A_1'&= \{ [\mu_1, \mu_2, \mu_1, \mu_1], [\mu_2, \mu_3, \mu_2, \mu_1], [\mu_2, \mu_3, \mu_1, \mu_2] \},
\\A_4'&= \{ [\mu_4, \mu_3, \mu_3, \mu_2], [\mu_4, \mu_3, \mu_2, \mu_3], [\mu_4, \mu_3, \mu_3, \mu_3], [\mu_2, \mu_3, \mu_2, \mu_2] \},
\\\overline{B}&= \{ [\mu_3, \mu_4, \mu_2, \mu_1] \}.
\end{align*}
In a similar way, we replace all commutators in the set $A_4'$ with the four commutators $[\mu_2, \mu_3, \mu_2, \mu_2], [\mu_3, \mu_4, \mu_3, \mu_3], [\mu_3, \mu_4, \mu_3, \mu_2], [\mu_3, \mu_4, \mu_2, \mu_3]$. We obtain the following set of commutators:
\begin{align*}
A_1'&= \{ [\mu_1, \mu_2, \mu_1, \mu_1], [\mu_2, \mu_3, \mu_2, \mu_1], [\mu_2, \mu_3, \mu_1, \mu_2] \},
\\A_4'&= \{ [\mu_2, \mu_3, \mu_2, \mu_2], [\mu_3, \mu_4, \mu_3, \mu_3], [\mu_3, \mu_4, \mu_3, \mu_2], [\mu_3, \mu_4, \mu_2, \mu_3] \},
\\\overline{B}&= \{ [\mu_3, \mu_4, \mu_2, \mu_1] \}.
\end{align*}
This set of commutators is minimal. Indeed, the assumption that some of the commutators can be
removed from the set $A_i'$ contradicts Theorem~\ref{commcox3} (because we would have three points for which the
system of generators is less than the minimal system in this theorem). Expanding the commutators
in $A_i'$, we see that they contain monomials with only two or three indices. Expanding the commutator
in $\overline{B}$ , we obtain monomials containing all four indices. It follows that in $L^4(\racg_\sK)$ it is impossible to
express the commutator in $\overline{B}$ via the commutators in $A_1'$ and $A_4'$. Indeed, the algebra is over $\mathbb Z_2$, and
should there exist an expression of this kind, it would be in the form of a sum of commutators, which
is impossible, because a sum of monomials containing less than four distinct indices cannot be equal
to a sum of monomials containing four distinct indices. It remains to prove that $[\mu_3, \mu_4, \mu_2, \mu_1] \neq 0$. Since $\racg_\sK'$ is a free group (see~\cite[Theorem 4.3]{pa-ve}) and hence all terms of the lower central series are free groups, to prove that a commutator is not zero, it suffices to show that the commutator corresponding
to it in $\gamma_2(\racg_\sK)$ is not zero and does not lie in $\gamma_5(\racg_\sK)$. Multiplying out, one can readily verify that $$(g_3, g_4, g_2, g_1) = (g_3, g_4, g_2)^{-1}  (g_3, g_4)^{-1} (g_1, g_2) (g_3, g_4) (g_3, g_4, g_2) (g_1, g_2)^{-1} \neq e$$ in $\gamma_2$, because this is the expression via the basis of $\gamma_2(\racg_\sK)$ (see~\cite[Theorem 4.5]{pa-ve}). Further, it does not lie in $\gamma_5$ (see Proposition~\ref{notgamma5elems2}). It follows that the commutator is nonzero in the algebra.

In view of the symmetry of the indices, we see that in the presence of the edges $\{i, j\}, \{j, k\}, \{k, l\}$ the set
\begin{align*}
A_k'&= \{ [\mu_k, \mu_i, \mu_k, \mu_k], [\mu_i, \mu_l, \mu_i, \mu_k], [\mu_i, \mu_l, \mu_k, \mu_i] \},
\\A_j'&= \{ [\mu_i, \mu_l, \mu_i, \mu_i], [\mu_l, \mu_j, \mu_l, \mu_l], [\mu_l, \mu_j, \mu_l, \mu_i], [\mu_l, \mu_j, \mu_i, \mu_l] \},
\\\overline{B}&= \{ [\mu_l, \mu_j, \mu_i, \mu_k] \}
\end{align*}
of commutators minimally generates $L^4(\racg_\sK)\cong\mathbb Z_2^8$.

In item (e), the following two versions of the simplicial complex (graph) are possible up to an isomorphism:
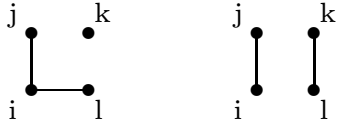
\begin{center}
\scalebox{1.5}{
\begin{picture}(10,10)
\put(0,2){\circle*{1}}
\put(5,2){\circle*{1}}
\put(0,7){\circle*{1}}
\put(5,7){\circle*{1}}
\put(0,2){\line(1,0){5}}
\put(0,2){\line(0,1){5}}
\put(-2,-0.5){\scriptsize i}
\put(-2,8){\scriptsize j}
\put(5.5,8){\scriptsize k}
\put(5.5,-0.5){\scriptsize l}
\end{picture}}\;\;\;\;\;\;\;\;\;\;\;\;
\scalebox{1.5}{
\begin{picture}(10,10)
\put(0,2){\circle*{1}}
\put(5,2){\circle*{1}}
\put(0,7){\circle*{1}}
\put(5,7){\circle*{1}}
\put(0,2){\line(0,1){5}}
\put(5,7){\line(0,-1){5}}
\put(-2,-0.5){\scriptsize i}
\put(-2,8){\scriptsize j}
\put(5.5,8){\scriptsize k}
\put(5.5,-0.5){\scriptsize l}
\end{picture}}
\end{center}

\vspace{0.1cm}
In the first case, we take $i = 4, j = 1, k = 2, l = 3$,
\begin{center}
\scalebox{1.5}{
\begin{picture}(6,7)
\put(0,2){\circle*{1}}
\put(5,2){\circle*{1}}
\put(0,7){\circle*{1}}
\put(5,7){\circle*{1}}
\put(0,2){\line(1,0){5}}
\put(0,2){\line(0,1){5}}
\put(-2,-0.5){\scriptsize 4}
\put(-2,8){\scriptsize 1}
\put(5.5,8){\scriptsize 2}
\put(5.5,-0.5){\scriptsize 3}
\end{picture}}
\end{center} remove all explicitly zero commutators from the set~\eqref{temp_setcomm4}and gather all commutators with indices in
$\{1, 2, 3\}$ into a set $A'_1$. The set of remaining commutators is
\begin{align*}
A_1'= \{&[\mu_3, \mu_2, \mu_2, \mu_1], [\mu_3, \mu_1, \mu_2, \mu_1], [\mu_3, \mu_1, \mu_1, \mu_2], [\mu_3, \mu_2, \mu_2, \mu_2],
\\
&[\mu_3, \mu_2, \mu_1, \mu_2], [\mu_3, \mu_1, \mu_2, \mu_3], [\mu_2, \mu_1, \mu_1, \mu_1], [\mu_3, \mu_1, \mu_1, \mu_1] \},
\\\overline{A_2}= \{&[\mu_4, \mu_2, \mu_2, \mu_1], [\mu_4, \mu_2, \mu_1, \mu_2] \},
\\\overline{A_4}= \{&[\mu_4, \mu_2, \mu_2, \mu_2], [\mu_4, \mu_2, \mu_3, \mu_2], [\mu_4, \mu_2, \mu_2, \mu_3], [\mu_4, \mu_2, \mu_3, \mu_4] \},
\\\overline{B}= \{&[\mu_2, \mu_4, \mu_3, \mu_1], [\mu_2, \mu_4, \mu_1, \mu_3] \}.
\end{align*}
The set $A_1'$ contains only eight commutators and is minimal by Theorem~\ref{commcox3}. We gather all commutators with indices in $\{1, 2, 4\}$ into a set $A_2'$,
\begin{align*}
A_1' = \{&[\mu_3, \mu_2, \mu_2, \mu_1], [\mu_3, \mu_1, \mu_2, \mu_1], [\mu_3, \mu_1, \mu_1, \mu_2], [\mu_3, \mu_2, \mu_2, \mu_2],
\\&[\mu_3, \mu_2, \mu_1, \mu_2], [\mu_3, \mu_1, \mu_2, \mu_3], [\mu_3, \mu_1, \mu_1, \mu_1] \},
\\A_2' = \{&[\mu_4, \mu_2, \mu_2, \mu_1], [\mu_4, \mu_2, \mu_1, \mu_2], [\mu_2, \mu_1, \mu_1, \mu_1], [\mu_4, \mu_2, \mu_2, \mu_2] \},
\\\overline{A_4} = \{&[\mu_4, \mu_2, \mu_3, \mu_2], [\mu_4, \mu_2, \mu_2, \mu_3], [\mu_4, \mu_2, \mu_3, \mu_4] \},
\\\overline{B} = \{&[\mu_2, \mu_4, \mu_3, \mu_1], [\mu_2, \mu_4, \mu_1, \mu_3] \}.
\end{align*}
The set $A_2'$ contains only four commutators and is minimal by Theorem~\ref{commcox3}. We gather all commutators with indices in $\{1, 3, 4\}$ into a set $A_3'$,
\begin{align*}
A_1' = \{&[\mu_3, \mu_2, \mu_2, \mu_1], [\mu_3, \mu_1, \mu_2, \mu_1], [\mu_3, \mu_1, \mu_1, \mu_2], [\mu_3, \mu_2, \mu_2, \mu_2],
\\&[\mu_3, \mu_2, \mu_1, \mu_2], [\mu_3, \mu_1, \mu_2, \mu_3] \},
\\A_2' = \{&[\mu_4, \mu_2, \mu_2, \mu_1], [\mu_4, \mu_2, \mu_1, \mu_2], [\mu_2, \mu_1, \mu_1, \mu_1], [\mu_4, \mu_2, \mu_2, \mu_2] \},
\\A_3' = \{&[\mu_3, \mu_1, \mu_1, \mu_1] \},
\\\overline{A_4} = \{&[\mu_4, \mu_2, \mu_3, \mu_2], [\mu_4, \mu_2, \mu_2, \mu_3], [\mu_4, \mu_2, \mu_3, \mu_4] \},
\\\overline{B} = \{&[\mu_2, \mu_4, \mu_3, \mu_1], [\mu_2, \mu_4, \mu_1, \mu_3] \}.
\end{align*}
The set $A_3'$ contains only one commutator and is minimal by Theorem~\ref{commcox3}. We gather all commutators with indices in $\{2, 3, 4\}$ into a set $A_4'$,
\begin{align*}
A_1' = \{&[\mu_3, \mu_2, \mu_2, \mu_1], [\mu_3, \mu_1, \mu_2, \mu_1], [\mu_3, \mu_1, \mu_1, \mu_2], [\mu_3, \mu_2, \mu_1, \mu_2],\\&[\mu_3, \mu_1, \mu_2, \mu_3] \},
\\A_2' = \{&[\mu_4, \mu_2, \mu_2, \mu_1], [\mu_4, \mu_2, \mu_1, \mu_2], [\mu_2, \mu_1, \mu_1, \mu_1] \},
\\A_3' = \{&[\mu_3, \mu_1, \mu_1, \mu_1] \},
\\A_4' = \{&[\mu_4, \mu_2, \mu_3, \mu_2], [\mu_4, \mu_2, \mu_2, \mu_3], [\mu_4, \mu_2, \mu_3, \mu_4], [\mu_4, \mu_2, \mu_2, \mu_2],\\&[\mu_3, \mu_2, \mu_2, \mu_2] \},
\\\overline{B} = \{&[\mu_2, \mu_4, \mu_3, \mu_1], [\mu_2, \mu_4, \mu_1, \mu_3] \}.
\end{align*}
By Theorem~\ref{commcox3}, all commutators in the set $A_4'$ can be replaced by the four commutators $[\mu_3, \mu_2, \mu_3, \mu_3], [\mu_2, \mu_4, \mu_2, \mu_2], [\mu_2, \mu_4, \mu_2, \mu_3], [\mu_2, \mu_4, \mu_3, \mu_2]$. We obtain the following set of commutators:
\begin{align*}
A_1' = \{&[\mu_3, \mu_2, \mu_2, \mu_1], [\mu_3, \mu_1, \mu_2, \mu_1], [\mu_3, \mu_1, \mu_1, \mu_2], [\mu_3, \mu_2, \mu_1, \mu_2],\\&[\mu_3, \mu_1, \mu_2, \mu_3] \},
\\A_2' = \{&[\mu_4, \mu_2, \mu_2, \mu_1], [\mu_4, \mu_2, \mu_1, \mu_2], [\mu_2, \mu_1, \mu_1, \mu_1] \},
\\A_3' = \{&[\mu_3, \mu_1, \mu_1, \mu_1] \},
\\A_4' = \{&[\mu_3, \mu_2, \mu_3, \mu_3], [\mu_2, \mu_4, \mu_2, \mu_2], [\mu_2, \mu_4, \mu_2, \mu_3], [\mu_2, \mu_4, \mu_3, \mu_2] \},
\\\overline{B} = \{&[\mu_2, \mu_4, \mu_3, \mu_1], [\mu_2, \mu_4, \mu_1, \mu_3] \}.
\end{align*}
This set of commutators is minimal, because the assumption that some of the commutators can be
removed from $A_i'$ contradicts Theorem~\ref{commcox3}. Expanding the commutators in $A_i'$, we see that they contain
monomials with only two or three indices. Expanding a commutator in $\overline{B}$, we see that the monomials
contain all four indices. It follows that in $L^4(\racg_\sK)$ any commutator in $\overline{B}$ cannot be expressed via the
other commutators, because the algebra is over $\mathbb Z_2$, and should there exist such an expression, it would
be a sum of other commutators, which is impossible because of the indices. It remains to prove that $$[\mu_2, \mu_4, \mu_3, \mu_1] \neq 0,~ [\mu_2, \mu_4, \mu_1, \mu_3] \neq 0.$$ Note that $\sK^1$ is a chordal graph and hence all $\gamma_s(\racg_\sK)$ are free groups (see~\cite[Theorem 4.3]{pa-ve}). By~\cite[Theorem 4.5]{pa-ve} the set
\begin{align*}
&(g_1,  g_3), (g_1,  g_2), (g_2,  g_3), (g_2,  g_4),
\\&(g_2, g_4, g_1), (g_2, g_4, g_3), (g_1, g_3, g_2), (g_2, g_3, g_1),
(g_2, g_4, g_3, g_1),
\end{align*}
is a basis of $\gamma_2(\racg_\sK)$,
whence it follows that $[\mu_2, \mu_4, \mu_3, \mu_1] \neq 0$, because the corresponding element in the group is not zero and does not lie in $\gamma_5(\racg_\sK)$ (see Proposition~\ref{notgamma5elems}). Further, note that the change $1 \leftrightarrow  3 $ preserves the structure of the graph, which means that $[\mu_2, \mu_4, \mu_1, \mu_3]$ is a nonzero element of $L^4(\racg_\sK)$ as well. It remains to prove that the commutators in $\overline{B}$ cannot be expressed via each other. We explicitly express $(g_2, g_4, g_1, g_3)$ via $(g_2, g_4, g_3, g_1)$ in $\gamma_2(\racg_\sK)$, using the second relation in~\ref{comtogPV}:
\begin{multline*}
    (g_2, g_4, g_1, g_3) = \\
    = (g_2, g_4, g_1)^{-1} (g_2, g_4, g_3)^{-1} (g_2, g_4)^{-1} (g_1, g_3)^{-1} (g_2, g_4) (g_2, g_4, g_1) (g_2, g_4, g_3) \\
    (g_2, g_4, g_3, g_1) (g_1, g_3) \equiv ((g_2, g_4), (g_1, g_3)) (g_2, g_4, g_3, g_1) \mod \gamma_5(\racg_\sK).
\end{multline*}
We see that these two commutators differ by $((g_2, g_4), (g_1, g_3))$, which is the commutator of two basis
elements in $\gamma_2(\racg_\sK)$. Since $\gamma_2(\racg_\sK)$ is a free group, it follows that this commutator is nonzero, and
further, it does not lie in $\gamma_5(\racg_\sK)$ (see Proposition~\ref{notgamma5elems}). Hence $$[\mu_2, \mu_4, \mu_3, \mu_1] \neq [\mu_2, \mu_4, \mu_1, \mu_3].$$ Should the commutators be expressible via each other, the relation $[\mu_2, \mu_4, \mu_3, \mu_1] = [\mu_2, \mu_4, \mu_1, \mu_3]^s$ would hold for some $s \in \mathbb N$, but since $L(\racg_\sK)$ is an algebra over $\mathbb Z_2$, we obtain zero for even $s$ and the
commutator itself for odd $s$, and they are not equal to each other as was shown above.

By the symmetry of the indices, we see that in the presence of the edges $\{i, j\}$ and $\{i, l\}$ the set
\begin{align*}
A_1' = \{&[\mu_l, \mu_k, \mu_k, \mu_j], [\mu_l, \mu_j, \mu_k, \mu_j], [\mu_l, \mu_j, \mu_j, \mu_k],
\\&[\mu_l, \mu_k, \mu_j, \mu_k], [\mu_l, \mu_j, \mu_k, \mu_l] \},
\\A_2' = \{&[\mu_i, \mu_k, \mu_k, \mu_j], [\mu_i, \mu_k, \mu_j, \mu_k], [\mu_k, \mu_j, \mu_j, \mu_j] \},
\\A_3' = \{&[\mu_l, \mu_j, \mu_j, \mu_j] \},
\\A_4' = \{&[\mu_l, \mu_k, \mu_l, \mu_l], [\mu_k, \mu_i, \mu_k, \mu_k], [\mu_k, \mu_i, \mu_k, \mu_l], [\mu_k, \mu_i, \mu_l, \mu_k] \},
\\\overline{B} = \{&[\mu_k, \mu_i, \mu_l, \mu_j], [\mu_k, \mu_i, \mu_j, \mu_l] \}
\end{align*}
of commutators minimally generates $L^4(\racg_\sK)\cong\mathbb Z_2^{15}$.

In the second case, we take $i = 1, j = 3, k = 2, l  =4$,
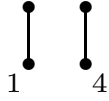
\begin{center}
\scalebox{1.5}{
\begin{picture}(6,7)
\put(0,2){\circle*{1}}
\put(5,2){\circle*{1}}
\put(0,7){\circle*{1}}
\put(5,7){\circle*{1}}
\put(0,2){\line(0,1){5}}
\put(5,7){\line(0,-1){5}}
\put(-2,-0.5){\scriptsize 1}
\put(-2,8){\scriptsize 3}
\put(5.5,8){\scriptsize 2}
\put(5.5,-0.5){\scriptsize 4}
\end{picture}}
\end{center} remove the commutators explicitly equal to zero from~\eqref{temp_setcomm4}, and gather all commutators with indices in $\{1, 2, 3\}$ into a set $A'_1$. The set of remaining commutators is
\begin{align*}
A_1' = \{&[\mu_3, \mu_2, \mu_2, \mu_2], [\mu_2, \mu_1, \mu_1, \mu_1], [\mu_3, \mu_2, \mu_2, \mu_1], [\mu_3, \mu_2, \mu_1, \mu_2] \},
\\\overline{A_2} = \{&[\mu_4, \mu_1, \mu_2, \mu_1] [\mu_4, \mu_1, \mu_1, \mu_2], [\mu_4, \mu_1, \mu_2, \mu_4] \},
\\\overline{A_3} = \{&[\mu_4, \mu_1, \mu_1, \mu_1], [\mu_4, \mu_3, \mu_3, \mu_1], [\mu_4, \mu_1, \mu_3, \mu_1], [\mu_4, \mu_1, \mu_1, \mu_3],
\\&[\mu_4, \mu_3, \mu_1, \mu_3], [\mu_4, \mu_1, \mu_3, \mu_4] \},
\\\overline{A_4} = \{&[\mu_4, \mu_3, \mu_3, \mu_2], [\mu_4, \mu_3, \mu_3, \mu_3], [\mu_4, \mu_3, \mu_2, \mu_3] \},
\\\overline{B} = \{&[\mu_1, \mu_4, \mu_3, \mu_2], [\mu_1, \mu_4, \mu_2, \mu_3], [\mu_3, \mu_4, \mu_1, \mu_2], [\mu_3, \mu_4, \mu_2, \mu_1] \}.
\end{align*}
By Theorem~\ref{commcox3}, all commutators in the set $A'_1$ can be replaced by the four commutators $[\mu_1, \mu_2, \mu_1, \mu_1], [\mu_2, \mu_3, \mu_2, \mu_2], [\mu_2, \mu_3, \mu_2, \mu_1], [\mu_2, \mu_3, \mu_1, \mu_2]$. After this, we gather all commutators with indices
in $\{1, 2, 4\}$ into a set $A'_2$,
\begin{align*}
A_1' = \{&[\mu_2, \mu_3, \mu_2, \mu_2], [\mu_2, \mu_3, \mu_2, \mu_1], [\mu_2, \mu_3, \mu_1, \mu_2] \},
\\A_2' = \{&[\mu_4, \mu_1, \mu_1, \mu_1], [\mu_1, \mu_2, \mu_1, \mu_1], [\mu_4, \mu_1, \mu_2, \mu_1], [\mu_4, \mu_1, \mu_1, \mu_2],\\&[\mu_4, \mu_1, \mu_2, \mu_4] \},
\\\overline{A_3} = \{&[\mu_4, \mu_3, \mu_3, \mu_1], [\mu_4, \mu_1, \mu_3, \mu_1], [\mu_4, \mu_1, \mu_1, \mu_3],
\\&[\mu_4, \mu_3, \mu_1, \mu_3], [\mu_4, \mu_1, \mu_3, \mu_4] \},
\\\overline{A_4} = \{&[\mu_4, \mu_3, \mu_3, \mu_2], [\mu_4, \mu_3, \mu_3, \mu_3], [\mu_4, \mu_3, \mu_2, \mu_3] \},
\\\overline{B} = \{&[\mu_1, \mu_4, \mu_3, \mu_2], [\mu_1, \mu_4, \mu_2, \mu_3], [\mu_3, \mu_4, \mu_1, \mu_2], [\mu_3, \mu_4, \mu_2, \mu_1] \}.
\end{align*}
In a similar way, all commutators in the set $A'_2$ can be replaced by the four commutators $[\mu_2, \mu_1, \mu_2, \mu_2], [\mu_1, \mu_4, \mu_1, \mu_1], [\mu_1, \mu_4, \mu_1, \mu_2], [\mu_1, \mu_4, \mu_2, \mu_1]$. After this, we gather all commutators with indices in $\{1, 3, 4\}$ into a set $A_3'$,
\begin{align*}
A_1' = \{&[\mu_2, \mu_3, \mu_2, \mu_2], [\mu_2, \mu_3, \mu_2, \mu_1], [\mu_2, \mu_3, \mu_1, \mu_2] \},
\\A_2' = \{&[\mu_2, \mu_1, \mu_2, \mu_2], [\mu_1, \mu_4, \mu_1, \mu_2], [\mu_1, \mu_4, \mu_2, \mu_1] \},
\\A_3' = \{&[\mu_4, \mu_3, \mu_3, \mu_3], [\mu_1, \mu_4, \mu_1, \mu_1], [\mu_4, \mu_3, \mu_3, \mu_1], [\mu_4, \mu_1, \mu_3, \mu_1],
\\&[\mu_4, \mu_1, \mu_1, \mu_3], [\mu_4, \mu_3, \mu_1, \mu_3], [\mu_4, \mu_1, \mu_3, \mu_4] \},
\\\overline{A_4} = \{&[\mu_4, \mu_3, \mu_3, \mu_2], [\mu_4, \mu_3, \mu_2, \mu_3] \},
\\\overline{B} = \{&[\mu_1, \mu_4, \mu_3, \mu_2], [\mu_1, \mu_4, \mu_2, \mu_3], [\mu_3, \mu_4, \mu_1, \mu_2], [\mu_3, \mu_4, \mu_2, \mu_1] \}.
\end{align*}
In a similar way, all commutators in the set $A'_3$ can be replaced by the four commutators $[\mu_1, \mu_4, \mu_1, \mu_1], [\mu_4, \mu_3, \mu_4, \mu_4], [\mu_4, \mu_3, \mu_4, \mu_1], [\mu_4, \mu_3, \mu_1, \mu_4]$. After this, we gather all commutators with indices
in $\{2, 3, 4\}$ into a set $A'_4$,
\begin{align*}
A_1' = \{&[\mu_2, \mu_3, \mu_2, \mu_1], [\mu_2, \mu_3, \mu_1, \mu_2] \},
\\A_2' = \{&[\mu_2, \mu_1, \mu_2, \mu_2], [\mu_1, \mu_4, \mu_1, \mu_2], [\mu_1, \mu_4, \mu_2, \mu_1] \},
\\A_3' = \{&[\mu_1, \mu_4, \mu_1, \mu_1], [\mu_4, \mu_3, \mu_4, \mu_1], [\mu_4, \mu_3, \mu_1, \mu_4] \},
\\A_4' = \{&[\mu_4, \mu_3, \mu_4, \mu_4], [\mu_2, \mu_3, \mu_2, \mu_2], [\mu_4, \mu_3, \mu_3, \mu_2], [\mu_4, \mu_3, \mu_2, \mu_3] \},
\\\overline{B} = \{&[\mu_1, \mu_4, \mu_3, \mu_2], [\mu_1, \mu_4, \mu_2, \mu_3], [\mu_3, \mu_4, \mu_1, \mu_2], [\mu_3, \mu_4, \mu_2, \mu_1] \}.
\end{align*}
In a similar way, all commutators in the set $A_4'$ can be replaced by the four commutators $[\mu_2, \mu_3, \mu_2, \mu_2], [\mu_3, \mu_4, \mu_3, \mu_3], [\mu_3, \mu_4, \mu_3, \mu_2], [\mu_3, \mu_4, \mu_2, \mu_3]$. We obtain the following set of commutators:
\begin{align*}
A_1' = \{&[\mu_2, \mu_3, \mu_2, \mu_1], [\mu_2, \mu_3, \mu_1, \mu_2] \},
\\
A_2' = \{&[\mu_2, \mu_1, \mu_2, \mu_2], [\mu_1, \mu_4, \mu_1, \mu_2], [\mu_1, \mu_4, \mu_2, \mu_1] \},
\\
A_3' = \{&[\mu_1, \mu_4, \mu_1, \mu_1], [\mu_4, \mu_3, \mu_4, \mu_1], [\mu_4, \mu_3, \mu_1, \mu_4] \},
\\
A_4' = \{&[\mu_2, \mu_3, \mu_2, \mu_2], [\mu_3, \mu_4, \mu_3, \mu_3], [\mu_3, \mu_4, \mu_3, \mu_2], [\mu_3, \mu_4, \mu_2, \mu_3] \},
\\
\overline{B} = \{&[\mu_1, \mu_4, \mu_3, \mu_2], [\mu_1, \mu_4, \mu_2, \mu_3], [\mu_3, \mu_4, \mu_1, \mu_2], [\mu_3, \mu_4, \mu_2, \mu_1] \}.
\end{align*}
We will show that $[\mu_3, \mu_4, \mu_1, \mu_2] = [\mu_1, \mu_4, \mu_3, \mu_2]$. Let us write the Jacobi identity for the ``outer'' commutators of the element $[\mu_3, \mu_4, \mu_2, \mu_1] = [[[\mu_3, \mu_4], \mu_2], \mu_1]$:
\begin{equation*}
[[[\mu_3, \mu_4], \mu_2], \mu_1] + [[\mu_2, \mu_1], [\mu_3, \mu_4]] + [[\mu_1, [\mu_3, \mu_4]], \mu_2] = 0.
\end{equation*}
This is equivalent to 
\begin{equation}\label{FC_eq1}
 [[\mu_4, \mu_3], [\mu_2, \mu_1]] = [[[\mu_3, \mu_4], \mu_2], \mu_1] + [[[\mu_3, \mu_4], \mu_1], \mu_2].
\end{equation}
In a similar way, from the commutator $[[[\mu_1, \mu_4], \mu_2], \mu_3]$ we obtain
\begin{equation}\label{FC_eq2}
[[\mu_2, \mu_3], [\mu_1, \mu_4]] = [[[\mu_1, \mu_4], \mu_2], \mu_3] + [[[\mu_1, \mu_4], \mu_3], \mu_2].
\end{equation}
We also write the Jacobi identity for the inner commutators of $[[[\mu_3, \mu_4], \mu_2], \mu_1]$,
\begin{equation*}
[[[\mu_3, \mu_4], \mu_2], \mu_1] + [[[\mu_4, \mu_2], \mu_3], \mu_1] + [[[\mu_2, \mu_3], \mu_4], \mu_1] = 0.
\end{equation*}
Since $[\mu_4, \mu_2] = 0$, we have
\begin{equation*}
[\mu_3, \mu_4, \mu_2, \mu_1] = [\mu_2, \mu_3, \mu_4, \mu_1].
\end{equation*}
Now we write the Jacobi identity for the ``inner'' commutators of the element $[\mu_2, \mu_3, \mu_4, \mu_1]$,
\begin{equation*}
[[[\mu_2, \mu_3], \mu_4], \mu_1] + [[\mu_4, \mu_1], [\mu_2, \mu_3]] + [[\mu_1, [\mu_2, \mu_3]], \mu_4] = 0.
\end{equation*}
From the last two identities, substituting \eqref{FC_eq2}, we have
\begin{equation}\label{FC_eq3}
[\mu_2, \mu_3, \mu_1, \mu_4] = [\mu_3, \mu_4, \mu_2, \mu_1] + [\mu_1, \mu_4, \mu_2, \mu_3] + [\mu_1, \mu_4, \mu_3, \mu_2].
\end{equation}
We carry out a similar argument for the ``inner'' commutators of $[\mu_1, \mu_4, \mu_2, \mu_3]$,
\begin{equation*}
[[[\mu_1, \mu_4], \mu_2], \mu_3] + [[[\mu_4, \mu_2], \mu_1], \mu_3] + [[[\mu_2, \mu_1], \mu_4], \mu_3] = 0 \Rightarrow
\end{equation*}
\begin{equation*}
\Rightarrow [\mu_1, \mu_4, \mu_2, \mu_3] = [\mu_2, \mu_1, \mu_4, \mu_3].
\end{equation*}
Consider the Jacobi identity
\begin{equation*}
 [[[\mu_2, \mu_1], \mu_4], \mu_3] + [[\mu_4, \mu_3], [\mu_2, \mu_1]] + [[[\mu_2, \mu_1], \mu_3], \mu_4] = 0.
\end{equation*}
From the last two identities, substituting \eqref{FC_eq1}, we have
\begin{equation}\label{FC_eq4}
[\mu_2, \mu_1, \mu_3, \mu_4] =  [\mu_1, \mu_4, \mu_2, \mu_3] + [\mu_3, \mu_4, \mu_2, \mu_1] + [\mu_3, \mu_4, \mu_1, \mu_2].
\end{equation}
Now note that the commutators $[\mu_2, \mu_1, \mu_3, \mu_4]$ and $[\mu_2, \mu_3, \mu_1, \mu_4]$ can readily be related to each other. Indeed, let us write the Jacobi identity for the ``inner'' commutators of $[\mu_2, \mu_1, \mu_3, \mu_4]$:
$$
[\mu_2, \mu_1, \mu_3, \mu_4] + [\mu_1, \mu_3, \mu_2, \mu_4] + [\mu_3, \mu_2, \mu_1, \mu_4] = 0.
$$
Since $[\mu_1, \mu_3] = 0$, it follows from this identity that
$
[\mu_2, \mu_1, \mu_3, \mu_4] = [\mu_2, \mu_3, \mu_1, \mu_4].
$
Substituting \eqref{FC_eq3} and \eqref{FC_eq4} into the preceding relation, we obtain
\begin{multline*}
[\mu_1, \mu_4, \mu_2, \mu_3] + [\mu_3, \mu_4, \mu_2, \mu_1] + [\mu_3, \mu_4, \mu_1, \mu_2]\\= [\mu_3, \mu_4, \mu_2, \mu_1] + [\mu_1, \mu_4, \mu_2, \mu_3] + [\mu_1, \mu_4, \mu_3, \mu_2],
\end{multline*}
whence
$$[\mu_3, \mu_4, \mu_1, \mu_2] = [\mu_1, \mu_4, \mu_3, \mu_2].$$
We obtain the new set
\begin{align*}
A_1' = \{&[\mu_2, \mu_3, \mu_2, \mu_1], [\mu_2, \mu_3, \mu_1, \mu_2] \},
\\
A_2' = \{&[\mu_2, \mu_1, \mu_2, \mu_2], [\mu_1, \mu_4, \mu_1, \mu_2], [\mu_1, \mu_4, \mu_2, \mu_1] \},
\\
A_3' = \{&[\mu_1, \mu_4, \mu_1, \mu_1], [\mu_4, \mu_3, \mu_4, \mu_1], [\mu_4, \mu_3, \mu_1, \mu_4] \},
\\
A_4' = \{&[\mu_2, \mu_3, \mu_2, \mu_2], [\mu_3, \mu_4, \mu_3, \mu_3], [\mu_3, \mu_4, \mu_3, \mu_2], [\mu_3, \mu_4, \mu_2, \mu_3] \},
\\
\overline{B} = \{&[\mu_1, \mu_4, \mu_2, \mu_3], [\mu_3, \mu_4, \mu_1, \mu_2], [\mu_3, \mu_4, \mu_2, \mu_1] \}.
\end{align*}
of commutators. This set is minimal by Proposition~\ref{numbergens4twoedges}. By the symmetry of the indices, we see that,
in the presence of the edges $\{i, j\}$, $\{k, l\}$, the set
\begin{align*}
A_i' = \{&[\mu_k, \mu_j, \mu_k, \mu_i], [\mu_k, \mu_j, \mu_i, \mu_k] \},
\\
A_k' = \{&[\mu_k, \mu_i, \mu_k, \mu_k], [\mu_i, \mu_l, \mu_i, \mu_k], [\mu_i, \mu_l, \mu_k, \mu_i] \},
\\
A_j' = \{&[\mu_i, \mu_l, \mu_i, \mu_i], [\mu_l, \mu_j, \mu_l, \mu_i], [\mu_l, \mu_j, \mu_i, \mu_l] \},
\\
A_l' = \{&[\mu_k, \mu_j, \mu_k, \mu_k], [\mu_j, \mu_l, \mu_j, \mu_j], [\mu_j, \mu_l, \mu_j, \mu_k], [\mu_j, \mu_l, \mu_k, \mu_j] \},
\\
\overline{B} = \{&[\mu_i, \mu_l, \mu_k, \mu_j], [\mu_j, \mu_l, \mu_i, \mu_k], [\mu_j, \mu_l, \mu_k, \mu_i] \}
\end{align*}
of commutators minimally generates $L^4(\racg_\sK)\cong\mathbb Z_2^{15}$.

In item (f), we assume that there is only one edge $\{i, j\}$. We take $i = 3, j = 4$,
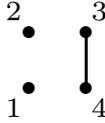
\begin{center}
\scalebox{1.5}{
\begin{picture}(6,7)
\put(0,2){\circle*{1}}
\put(5,2){\circle*{1}}
\put(0,7){\circle*{1}}
\put(5,7){\circle*{1}}
\put(5,7){\line(0,-1){5}}
\put(-2,-0.5){\scriptsize 1}
\put(-2,8){\scriptsize 2}
\put(5.5,8){\scriptsize 3}
\put(5.5,-0.5){\scriptsize 4}
\end{picture}}
\end{center} remove the commutators explicitly equal to zero from~\eqref{temp_setcomm4}, and gather all commutators with indices in
$\{1, 2, 3\}$ into a set $A'_1$. The set of remaining commutators is
\begin{align*}
A_1' = \{&[\mu_3, \mu_2, \mu_2, \mu_1], [\mu_3, \mu_1, \mu_2, \mu_1], [\mu_3, \mu_1, \mu_1, \mu_2], [\mu_3, \mu_2, \mu_1, \mu_2],
\\&
[\mu_2, \mu_1, \mu_1, \mu_1], [\mu_3, \mu_1, \mu_2, \mu_3], [\mu_3, \mu_2, \mu_2, \mu_2], [\mu_3, \mu_1, \mu_1, \mu_1] \},
\\
\overline{A_2} = \{&[\mu_4, \mu_2, \mu_2, \mu_1], [\mu_4, \mu_1, \mu_2, \mu_1], [\mu_4, \mu_1, \mu_1, \mu_2],
\\&
[\mu_4, \mu_2, \mu_1, \mu_2], [\mu_4, \mu_1, \mu_2, \mu_4] \},
\\
\overline{A_3} = \{&[\mu_4, \mu_1, \mu_1, \mu_1], [\mu_4, \mu_1, \mu_3, \mu_1], [\mu_4, \mu_1, \mu_1, \mu_3], [\mu_4, \mu_1, \mu_3, \mu_4] \},
\\
\overline{A_4} = \{&[\mu_4, \mu_2, \mu_2, \mu_2], [\mu_4, \mu_2, \mu_3, \mu_2], [\mu_4, \mu_2, \mu_2, \mu_3], [\mu_4, \mu_2, \mu_3, \mu_4] \},
\\
\overline{B} = \{&[\mu_2, \mu_4, \mu_3, \mu_1], [\mu_1, \mu_4, \mu_3, \mu_2], [\mu_1, \mu_4, \mu_2, \mu_3], [\mu_2, \mu_4, \mu_1, \mu_3] \},
\end{align*}
By Corollary~\ref{cor_freeproduct3}, all commutators in the set $A'_1$ can be replaced by the eight commutators
\begin{align*}
&[\mu_2, \mu_1, \mu_1, \mu_1], [\mu_3, \mu_1, \mu_1, \mu_1], [\mu_3, \mu_2, \mu_2, \mu_1], [\mu_3, \mu_1, \mu_2, \mu_1],
\\&[\mu_3, \mu_1, \mu_1, \mu_2],[\mu_3, \mu_2, \mu_2, \mu_2], [\mu_3, \mu_2, \mu_1, \mu_2], [\mu_3, \mu_1, \mu_2, \mu_3].
\end{align*}
After this, we gather all commutators with indices $\{1, 2, 4\}$ into a set $A'_2$,
\begin{align*}
A_1' = \{&[\mu_3, \mu_1, \mu_1, \mu_1], [\mu_3, \mu_2, \mu_2, \mu_1], [\mu_3, \mu_1, \mu_2, \mu_1], [\mu_3, \mu_1, \mu_1, \mu_2],
\\&
[\mu_3, \mu_2, \mu_2, \mu_2], [\mu_3, \mu_2, \mu_1, \mu_2], [\mu_3, \mu_1, \mu_2, \mu_3] \},
\\
A_2' = \{&[\mu_4, \mu_1, \mu_1, \mu_1], [\mu_4, \mu_2, \mu_2, \mu_2], [\mu_2, \mu_1, \mu_1, \mu_1], [\mu_4, \mu_2, \mu_2, \mu_1],
\\&
[\mu_4, \mu_1, \mu_2, \mu_1], [\mu_4, \mu_1, \mu_1, \mu_2], [\mu_4, \mu_2, \mu_1, \mu_2], [\mu_4, \mu_1, \mu_2, \mu_4] \},
\\
\overline{A_3} = \{&[\mu_4, \mu_1, \mu_3, \mu_1], [\mu_4, \mu_1, \mu_1, \mu_3], [\mu_4, \mu_1, \mu_3, \mu_4] \},
\\
\overline{A_4} = \{&[\mu_4, \mu_2, \mu_3, \mu_2], [\mu_4, \mu_2, \mu_2, \mu_3], [\mu_4, \mu_2, \mu_3, \mu_4] \},
\\
\overline{B} = \{&[\mu_2, \mu_4, \mu_3, \mu_1], [\mu_1, \mu_4, \mu_3, \mu_2], [\mu_1, \mu_4, \mu_2, \mu_3], [\mu_2, \mu_4, \mu_1, \mu_3] \},
\end{align*}
In a similar way, all commutators in the set $A'_2$ can be replaced by the eight commutators
\begin{align*}
[\mu_2, \mu_1, \mu_1, \mu_1], [\mu_4, \mu_1, \mu_1, \mu_1], [\mu_4, \mu_2, \mu_2, \mu_1], [\mu_4, \mu_1, \mu_2, \mu_1],\\
[\mu_4, \mu_1, \mu_1, \mu_2], [\mu_4, \mu_2, \mu_2, \mu_2], [\mu_4, \mu_2, \mu_1, \mu_2], [\mu_4, \mu_1, \mu_2, \mu_4].
\end{align*}
After this we gather all commutators with indices in $\{1, 3, 4\}$ into a set $A_3'$,
\begin{align*}
A_1' = \{&[\mu_3, \mu_2, \mu_2, \mu_1], [\mu_3, \mu_1, \mu_2, \mu_1], [\mu_3, \mu_1, \mu_1, \mu_2],
\\&
[\mu_3, \mu_2, \mu_2, \mu_2], [\mu_3, \mu_2, \mu_1, \mu_2], [\mu_3, \mu_1, \mu_2, \mu_3] \},
\\
A_2' = \{&[\mu_2, \mu_1, \mu_1, \mu_1], [\mu_4, \mu_2, \mu_2, \mu_1], [\mu_4, \mu_1, \mu_2, \mu_1], [\mu_4, \mu_1, \mu_1, \mu_2],
\\&
[\mu_4, \mu_2, \mu_2, \mu_2], [\mu_4, \mu_2, \mu_1, \mu_2], [\mu_4, \mu_1, \mu_2, \mu_4] \},
\\
A_3' = \{&[\mu_3, \mu_1, \mu_1, \mu_1], [\mu_4, \mu_1, \mu_1, \mu_1], [\mu_4, \mu_1, \mu_3, \mu_1], [\mu_4, \mu_1, \mu_1, \mu_3],\\&[\mu_4, \mu_1, \mu_3, \mu_4] \},
\\
\overline{A_4} = \{&[\mu_4, \mu_2, \mu_3, \mu_2], [\mu_4, \mu_2, \mu_2, \mu_3], [\mu_4, \mu_2, \mu_3, \mu_4] \},
\\
\overline{B} = \{&[\mu_2, \mu_4, \mu_3, \mu_1], [\mu_1, \mu_4, \mu_3, \mu_2], [\mu_1, \mu_4, \mu_2, \mu_3], [\mu_2, \mu_4, \mu_1, \mu_3] \}.
\end{align*}
By Theorem~\ref{commcox3}, all commutators in the set $A'_3$ can be replaced by the four commutators $[\mu_3, \mu_1, \mu_3, \mu_3], [\mu_1, \mu_4, \mu_1, \mu_1], [\mu_1, \mu_4, \mu_1, \mu_3], [\mu_1, \mu_4, \mu_3, \mu_1]$. After this we gather all commutators with indices in $\{2, 3, 4\}$ into a set $A'_4$:
\begin{align*}
A_1' = \{&[\mu_3, \mu_2, \mu_2, \mu_1], [\mu_3, \mu_1, \mu_2, \mu_1], [\mu_3, \mu_1, \mu_1, \mu_2], [\mu_3, \mu_2, \mu_1, \mu_2],\\&[\mu_3, \mu_1, \mu_2, \mu_3] \},\\
A_2' = \{&[\mu_2, \mu_1, \mu_1, \mu_1], [\mu_4, \mu_2, \mu_2, \mu_1], [\mu_4, \mu_1, \mu_2, \mu_1], [\mu_4, \mu_1, \mu_1, \mu_2],
\\&
[\mu_4, \mu_2, \mu_1, \mu_2], [\mu_4, \mu_1, \mu_2, \mu_4] \},
\\
A_3' = \{&[\mu_3, \mu_1, \mu_3, \mu_3], [\mu_1, \mu_4, \mu_1, \mu_1], [\mu_1, \mu_4, \mu_1, \mu_3], [\mu_1, \mu_4, \mu_3, \mu_1] \},
\\
A_4' = \{&[\mu_3, \mu_2, \mu_2, \mu_2], [\mu_4, \mu_2, \mu_2, \mu_2], [\mu_4, \mu_2, \mu_3, \mu_2], [\mu_4, \mu_2, \mu_2, \mu_3],\\&[\mu_4, \mu_2, \mu_3, \mu_4] \},
\\
\overline{B} = \{&[\mu_2, \mu_4, \mu_3, \mu_1], [\mu_1, \mu_4, \mu_3, \mu_2], [\mu_1, \mu_4, \mu_2, \mu_3], [\mu_2, \mu_4, \mu_1, \mu_3] \}.
\end{align*}
In a similar way, all commutators in the set $A_4'$ can be replaced by the four commutators $[\mu_3, \mu_2, \mu_3, \mu_3], [\mu_2, \mu_4, \mu_2, \mu_2], [\mu_2, \mu_4, \mu_2, \mu_3], [\mu_2, \mu_4, \mu_3, \mu_2]$. We obtain the following set of commutators:
\begin{align*}
A_1' = \{&[\mu_3, \mu_2, \mu_2, \mu_1], [\mu_3, \mu_1, \mu_2, \mu_1], [\mu_3, \mu_1, \mu_1, \mu_2], [\mu_3, \mu_2, \mu_1, \mu_2],\\&[\mu_3, \mu_1, \mu_2, \mu_3] \},
\\
A_2' = \{&[\mu_2, \mu_1, \mu_1, \mu_1], [\mu_4, \mu_2, \mu_2, \mu_1], [\mu_4, \mu_1, \mu_2, \mu_1], [\mu_4, \mu_1, \mu_1, \mu_2],\\&
[\mu_4, \mu_2, \mu_1, \mu_2],
[\mu_4, \mu_1, \mu_2, \mu_4] \},
\\
A_3' = \{&[\mu_3, \mu_1, \mu_3, \mu_3], [\mu_1, \mu_4, \mu_1, \mu_1], [\mu_1, \mu_4, \mu_1, \mu_3], [\mu_1, \mu_4, \mu_3, \mu_1] \},
\\
A_4' = \{&[\mu_3, \mu_2, \mu_3, \mu_3], [\mu_2, \mu_4, \mu_2, \mu_2], [\mu_2, \mu_4, \mu_2, \mu_3], [\mu_2, \mu_4, \mu_3, \mu_2] \},
\\
\overline{B} = \{&[\mu_2, \mu_4, \mu_3, \mu_1], [\mu_1, \mu_4, \mu_3, \mu_2], [\mu_1, \mu_4, \mu_2, \mu_3], [\mu_2, \mu_4, \mu_1, \mu_3] \}.
\end{align*}
This set of commutators is minimal by Proposition~\ref{numbergens4oneedge}. By the symmetry of the indices, we see that in
the presence of the edge $\{i, j\}$ the set
\begin{align*}
A_k = \{&[\mu_i, \mu_l, \mu_l, \mu_k], [\mu_i, \mu_k, \mu_l, \mu_k], [\mu_i, \mu_k, \mu_k, \mu_l], [\mu_i, \mu_l, \mu_k, \mu_l],\\&
[\mu_i, \mu_k, \mu_l, \mu_i] \},
\\
A_l'= \{&[\mu_l, \mu_k, \mu_k, \mu_k], [\mu_j, \mu_l, \mu_l, \mu_k], [\mu_j, \mu_k, \mu_l, \mu_k], [\mu_j, \mu_k, \mu_k, \mu_l],\\
&[\mu_j, \mu_l, \mu_k, \mu_l], [\mu_j, \mu_k, \mu_l, \mu_j] \},
\\
A_i'= \{&[\mu_i, \mu_k, \mu_i, \mu_i], [\mu_k, \mu_j, \mu_k, \mu_k], [\mu_k, \mu_j, \mu_k, \mu_i], [\mu_k, \mu_j, \mu_i, \mu_k] \},
\\A_j'= \{&[\mu_i, \mu_l, \mu_i, \mu_i], [\mu_l, \mu_j, \mu_l, \mu_l], [\mu_l, \mu_j, \mu_l, \mu_i], [\mu_l, \mu_j, \mu_i, \mu_l] \},
\\\overline{B}= \{&[\mu_l, \mu_j, \mu_i, \mu_k], [\mu_k, \mu_j, \mu_i, \mu_l], [\mu_k, \mu_j, \mu_l, \mu_i], [\mu_l, \mu_j, \mu_k, \mu_i] \}
\end{align*}
of commutators minimally generates $L^4(\racg_\sK)\cong\mathbb Z_2^{23}$.
\end{proof}

\section{Generators of $L^4(\racg_\sK)$}

Using Proposition~\ref{coxeter4pointsgens}, we can write the generators of $L(\racg_\sK)$ for any simplicial complex $\sK$ on $[m]$. The algorithm is similar to the proof of some items of Proposition~\ref{coxeter4pointsgens}:
\begin{enumerate}
\item Using Theorem~\ref{theorem_freeproduct4}, we write the generators for the discrete simplcial complex $\sK'$ on $m$ points.
\item We take four points $J \subset [m]$ and then isolate the restriction $\sK_J$, gathering all commutators
corresponding to it into a set $A$.
\item Using Proposition~\ref{coxeter4pointsgens}, we write the minimal set $A'$ of generators for the resulting complex.
\item We return the set $A'$ to the general list of commutators instead of $A$.
\item We return to the first item of the list, take some other four points, and repeat this operation for
each quadruple of points (thus, the operation must be performed $\binom{m}{4}$ times).
\end{enumerate}

\begin{proof}
At step $3$, by Proposition~\ref{coxeter4pointsgens}, the subset $A'$ combined with the complement of $A$ generates $L^4(\racg_\sK)$, whence it follows that the resulting set generates $L^4(\racg_\sK)$. It remains to prove the minimality. Assume that the resulting set of commutators is not a minimal system of generators;
i.e., there exists a commutator $\alpha = [\mu_i, \mu_j, \mu_k, \mu_l]$ which can be removed and the remaining set still
generates $L^4(\racg_\sK)$. We consider the iteration of the algorithm when $\alpha$ was considered for the last time at step $4$ and take the same four points that were used for this iteration. By Proposition~\ref{coxeter4pointsgens}, the commutator $\alpha$ cannot be removed, which is a contradiction. It follows that the algorithm gives the desired result.
\end{proof}

\end{document}